\documentclass[11pt]{article}
\usepackage{amsmath,amsthm,amssymb,amscd,fancybox,ifthen,epsfig,subfigure}
\usepackage[all]{xy}
\usepackage{times}
\usepackage{hyperref} 

\newcommand{\mar}[1]{{\marginpar{\sffamily{\scriptsize #1}}}}
\newcommand\CLE[1]{\mar{{\bf CLE}: #1}}

\newcommand\loc{\operatorname{loc}}

\newcommand\Dom{\operatorname{Dom}}
\newcommand\Div{\operatorname{Div}}

\newcommand\dist{\operatorname{d}}

\newcommand\Int{\operatorname{int}}

\newcommand\Tr{\operatorname{tr}}

\newcommand\Spec{\operatorname{spec}}

\newcommand\cI{\mathcal I}

\newcommand\cU{\mathcal U}

\newcommand\tx{\tilde x}
\newcommand\tw{\tilde w}

\newcommand\ty{\tilde y}

\newcommand\ub{\underline{b}}
\newcommand\ubb{\boldsymbol{\underline{b}}}

\newcommand\bb{\boldsymbol b}

\newcommand\tchi{\widetilde\chi}

\newcommand\hV{\widehat{V}}
\newcommand\hL{\widehat{L}}

\newcommand\hc{\widehat{c}}

\newcommand\bx{\boldsymbol x}

\newcommand\by{\boldsymbol y}

\newcommand\bone{\boldsymbol e}

\newcommand\cF{\mathcal F}

\newcommand\tb{\widetilde{b}}

\newcommand\tu{\widetilde{u}}
\newcommand\tQ{\widetilde{Q}}

\newcommand\pa{\partial}

\newcommand\cC{\mathcal{C}}
\newcommand\cL{\mathcal{L}}

\newcommand\cD{\mathcal{D}}

\renewcommand\Re{\operatorname{Re}}

\newcommand\bbN{\mathbb N}

\newcommand\bbR{\mathbb R}

\newcommand\restrictedto{\upharpoonright}

\newcommand\supp{\operatorname{supp}}

\newcommand\CI{{\mathcal C}^{\infty}}

\newcommand\CIc{{\mathcal C}^{\infty}_{\text{c}}}

\DeclareMathOperator{\WF}{WF}

\newtheorem{theorem}{Theorem}[section]
\newtheorem{proposition}{Proposition}[section]
\newtheorem{corollary}{Corollary}[section]
\newtheorem{lemma}{Lemma}[section]

\theoremstyle{definition}
\newtheorem{definition}{Definition}[section]

\theoremstyle{remark}
\newtheorem{remark}{Remark}[section]

\begin{document}

\title{Harnack Inequalities and Heat-kernel Estimates for Degenerate Diffusion
  Operators Arising in Population Biology}

\author{Charles L. Epstein\footnote{Research partially supported by
    NSF grant DMS12-05851, and  ARO grant W911NF-12-1-0552.
\newline
Address: Department of Mathematics, University of
    Pennsylvania; e-mail: cle@math.upenn.edu }\quad
and Rafe Mazzeo\footnote{Research partially supported by
    NSF grant DMS1105050. Address:
Department of Mathematics, Stanford
    University; e-mail: mazzeo@math.stanford.edu\newline
MSC-2010: 35K65, 35K08, 35B65, 35P20, 35Q92, 60J60, 92D10, 42B37}}
\date{August 10, 2014}

\maketitle
\noindent

\medskip
\begin{abstract} 
  This paper continues the analysis, started in~\cite{WF1d,EpMaz2}, of a class of
  degenerate elliptic operators defined on manifolds with corners, which arise
  in Population Biology. Using techniques pioneered by J. Moser, and extended
  and refined by L. Saloff-Coste, Grigor'yan, and Sturm, we show that weak
  solutions to the parabolic problem defined by a sub-class of these operators,
  which consists of those that can be defined by Dirichlet forms and have
  non-vanishing transverse vector field, satisfy a Harnack inequality. This
  allows us to conclude that the solutions to these equations belong, for
  positive times, to the natural anisotropic H\"older spaces, and also leads to
  upper and, in some cases, lower bounds for the heat kernels of these
  operators. These results imply that these operators have a compact resolvent
  when acting on $\cC^0$ or $L^2.$ The proof relies upon a scale invariant
  Poincar\'e inequality that we establish for a large class of weighted
  Dirichlet forms, as well as estimates to handle certain mildly singular
  perturbation terms. The weights that we consider are neither Ahlfors regular,
  nor do they generally belong to the Muckenhaupt class $A_2.$
  \end{abstract}

\section{Introduction}
In a series of paper and a book we have considered the analysis of a class of
degenerate diffusion operators, which arise in Population Biology,
see~\cite{Kimura1964}, which we call generalized Kimura diffusion
operators. The typical examples that arise in population
genetics act on functions defined on the $n$-simplex
\begin{equation}
  \Sigma_n=\{(x_1,\dots,x_n):\: 0\leq x_j\text{ and }x_1+\cdots +x_n\leq 1\},
\end{equation}
and take the form
\begin{equation}
  L=\sum_{i,j=1}^n(x_i\delta_{ij}-x_ix_j)\pa_{x_i}\pa_{x_j}+\sum_{i=1}^{n}b_i(x)\pa_{x_i}.
\end{equation}
The vector field is inward pointing, and the coefficient functions $\{b_i(x)\}$
are often either linear or quadratic polynomials.  The class of operators we
analyze includes these examples, but is considerably more general. They are
defined on manifolds with corners by degenerate, elliptic, partial differential
operators. In ``adapted local coordinates'' $(x;y)\in
S_{n,m}=\bbR_+^n\times\bbR^m,$ (where $\bbR_+=[0,\infty)$), such an operator
takes the form:
\begin{multline}\label{eqn0.001}
  Lu=\sum_{j=1}^n[x_i\pa_{x_i}^2+b_i(x;y)\pa_{x_i}]u+\sum_{i,j=1}^n+
x_ix_ja_{ij}(x;y)\pa_{x_i}\pa_{x_j}u+\\
\sum_{i=1}^n\sum_{l=1}^mx_ic_{il}\pa_{x_i}\pa_{y_l}u+
\sum_{k,l=1}^md_{kl}(x;y)\pa_{y_k}\pa_{y_l}u+\sum_{l=1}^md_l(x;y)\pa_{y_l}u,
\end{multline}
in a neighborhood of $(0;0).$ 

In our work thus far we have assumed that the coefficients
are smooth functions of the variables $(\bx;\by),$ or of the ``square root'' variables,
$$(\sqrt{\bx};\by)\overset{d}{=}(\sqrt{x_1},\dots,\sqrt{x_n};y_1,\dots.y_m);$$
later in this paper we  see that somewhat less
regular coefficients arise naturally.  The monograph~\cite{EpMaz2} provides a starting
point for the analysis of generalized Kimura diffusion operators by analyzing
the so-called ``backward Kolmogorov'' operator acting on data belonging to a family of anisotropic
H\"older spaces. Central to this study are the explicit heat kernels associated to the model operators
\begin{equation}\label{eqn2.009}
  L_{\bb,m}=\sum_{j=1}^n[x_i\pa_{x_i}^2+b_i\pa_{x_i}]+\sum_{l=1}^m\pa_{y_l}^2,
\end{equation}
acting on functions defined on $S_{n,m}.$ These kernels are used to construct
parametrices for the heat and resolvent kernels for a generalized Kimura
diffusion operator on a compact manifold with corners. This parametrix
construction is far from sharp, but using it and various functional analytic
arguments connected to anisotropic H\"older spaces, we establish existence,
uniqueness and essentially optimal regularity results in this setting. This
leads to a proof of existence of the Feller semigroup acting on $\cC^0,$ which
is of importance in biological applications, but it is not informative as to
the regularity properties of solutions to the parabolic problem with merely
continuous initial data. 

This parametrix approach does not give optimal regularity results for solutions with initial data in
$\cC^0,$ or regularity results for local solutions, nor does it lead
to pointwise estimates for the heat kernel.  For many applications, such
heat kernel estimates and local regularity results are quite important, which
has motivated our further work on this problem.  One step was taken
in~\cite{EpMaz3}, where we treated the special case where $P$ is a manifold
with boundary.  In that setting we were able to adapt the techniques of
geometric microlocal analysis to give more precise information on the heat
kernel, which then directly implies the various optimal regularity results for
solutions of the heat equation, including the precise regularity for solutions
with initial data in $\cC^0$.

In the present paper we continue this  program in a somewhat different setting
using very different techniques. We use the formalism of Dirichlet forms, weak solutions, 
and Moser's approach to Harnack inequalities, as clarified and extended by Saloff-Coste, 
Grigor'yan, and Sturm, see~\cite{Moser1, Moser2, SaloffCosteLMS} and~\cite{Sturm1, Sturm2, Sturm3},
to prove that local solutions of the parabolic equations associated to certain generalized Kimura
diffusion operators satisfy a Harnack inequality. We also adapt the results from the papers
just cited to explain how this leads to upper and (sometimes) lower pointwise bounds for the heat kernel, and
H\"older regularity at positive times for local, weak solutions of the Cauchy problem. 

The analysis in this paper brings to the fore the mutation rates, which, in the
mathematical formulation, appear as normalized coefficients of a vector field
transverse to the boundary that we call \emph{weights.} These are essentially  the
functions $\{b_i(x;y)\}$ appearing in~\eqref{eqn0.001}, restricted
to the respective subsets of $\pa P,$ given by $\{x_i=0\}.$ In~\cite{EpMaz2} no
hypothesis is made on the weights, other than non-negativity, though it has
been apparent for some time that the structure of the heat kernel is radically
different along the part of the boundary where weights vanish. An early result
along these lines in given in~\cite{shimakura2}.

In the Dirichlet form approach the weights define a measure given locally by
\begin{equation}\label{eqn3.004}
  d\mu_{\bb}(x;y)=e^U(x;y)x_1^{b_1(x;y)-1}\cdots x_n^{b_n(x;y)-1}dx_1\cdots dx_n\cdot
  dy_1\cdots dy_m
\end{equation}
on a neighborhood of $(0,0)$ in $S_{n,m}.$ Here $U$ is a bounded function,
which we take to be zero for the remainder of the introduction. The second
order part of the generalized Kimura diffusion operator, i.e.~the principal
symbol, defines a quadratic form on functions in $\cC^1_c(S_{n,m})$ of the form
\begin{multline}
  q(u,v)(x;y)=\bigg(\sum_{j=1}^n[x_i\pa_{x_i}u\pa_{x_i}v+\sum_{i,j=1}^n+
x_ix_ja_{ij}(x;y)\pa_{x_i}u\pa_{x_j}v+\\
\frac{1}{2}\sum_{i=1}^n\sum_{l=1}^mx_ic_{il}[\pa_{x_i}u\pa_{y_l}v+\pa_{x_i}v\pa_{y_l}u]+
\sum_{k,l=1}^md_{kl}(x;y)\pa_{y_k}u\pa_{y_l}v\bigg)(x;y).
\end{multline}
In the body of the paper, this is abbreviated as
\begin{equation}
   q(u,v)(x;y)=\langle A(x;y)\nabla u(x;y),\nabla v(x;y)\rangle.
\end{equation}
The measure $d\mu_{\bb}$ and the quadratic form $q$ together define a Dirichlet form
\begin{equation}\label{eqn1.009}
  Q(u,v)=\int \langle A(x;y)\nabla u(x;y),\nabla v(x;y)\rangle d\mu_{\bb}(x;y).
\end{equation}
Formally integrating by parts, we arrive at an operator, $L_Q,$ with a densely defined
domain in $L^2(S_{n,m};d\mu_{\bb}),$ specified by a ``natural'' boundary
condition. On sufficiently smooth initial data, it is easy to check that the
solution of the parabolic problem defined by this operator
agrees with the regular solution for ``backward Kolmogorov'' operator analyzed
in~\cite{EpMaz2}. Similar considerations apply to define self adjoint operators
on $L^2(B;d\mu_{\bb})$ for open sets $B\subset S_{n,m}.$ 

It is well known that $Q$ can be modified by the addition of a non-symmetric
term
\begin{equation}
  \int \langle A(x;y)\nabla u(x;y),X(x;y) v(x;y)\rangle d\mu_{\bb}(x;y),
\end{equation}
where $X(x;y)$ is an $\bbR^{n+m}$-valued function. This has the effect of adding a
tangential vector field, $V_X,$ to $L_Q.$ If a weight $b_j(x;y)$ is 
non-constant along a portion of the boundary where $x_j=0,$ then $L_Q$ includes 
a vector field tangent to this boundary hypersurface, with mildly singular coefficients of the form
\begin{equation}
  \sum_{i,j}\alpha_{ij}(x;y)\log x_jx_i\pa_{x_i}+\sum_{j,l}\beta_{j,l}(x;y)\log x_j\pa_{y_l}.
\end{equation}
These terms do not appear if the weights $\{b_i(x;y)\}$ are constant along the appropriate
boundary components. In other words, to obtain an arbitrary generalized Kimura
diffusion operator, as in~\cite{EpMaz2}, using a Dirichlet form we must allow coefficient
functions $X$ with log singularities, i.e., which satisfy
\begin{equation}
  \langle A(x;y)X(x;y),X(x;y) \rangle\leq M\left[\sum_{i=1}^n|\log x_i|+1\right]^2,
\end{equation}
near the boundary of $S_{n,m}.$ 

If $L$ is a generalized Kimura diffusion operator with smooth coefficients,
acting as an unbounded operator on $\cC^0$-functions, then its dual $L^t$ acts
naturally on a subspace of the regular Borel measures. Since $\cC^0$ is
non-reflexive, the obvious dual semi-group is not strongly continuous as $t\to
0^+.$ Following Phillips, one restricts to a subspace on which it is. Measures
belonging to this subspace are absolutely continuous, away from the boundary,
with respect to $d\mu_{\bb}.$ For this, and other reasons (see
Theorem~\ref{cor5.2.22}) it is natural to represent them in the form
$wd\mu_{\bb}.$ If the weights are non-constant, then the differential operator
representing the action of $L^t$ on $w$ has lower order terms with
logarithmically singular coefficients. Such terms are therefore not simply an
artifact of our method, but rather intrinsic to this class of operators. These
types of singular terms can be controlled using several variants of the following lemma:
\begin{lemma}\label{lemB2.2.01010}[ See Lemma~\ref{lemB2.2.010}] Assume that
  $\bb=(b_1,\dots,b_n)$ are positive differentiable functions of $(\bx;\by),$
  with $0<\beta_0<b_j,$ constant outside a compact set. Let $q$ be a measurable
  function on $S_{n,m}$ that satisfies
\begin{equation}\label{eqn220.0121}
    |q(x;y)|\leq M\left[\chi_B(x;y)\sum_{j=1}^l|\log x_i|^k+1\right],
  \end{equation}
  for some $k\in\bbN,$ $B$ a bounded set, and $M > 0.$ Given $\eta>0$ there
  is a $0<\delta<\frac 12,$ so that if $\supp\chi\subset [0,\delta]^n\times
  (-1,1)^m,$ then there is a $C_{\eta}$ so that
\begin{multline}
  \int\limits_{S_{n,m}}\chi^2(x;y)|q(x;y)|u^2(x;y)d\mu_{\bb}\leq
 \eta\int\limits_{S_{n,m}}\langle A\nabla u,\nabla
 u\rangle\chi^2d\mu_{\bb}+\\
C_{\eta} \int\limits_{S_{n,m}}[\langle A\nabla \chi,\nabla\chi\rangle+\chi^2]u^2d\mu_{\bb},
\end{multline}
for any positive differentiable function $u$. 
\end{lemma}
\begin{remark} This allows us to control the singular terms in a neighborhood
  of any boundary point. Since the weights $\bb$ are constant outside of a
  compact set, we can use this lemma along with a simple covering argument to
  show that these singular terms are bounded by a small multiple of $Q(u,u),$
  plus a large multiple of the $L^2$-norm of $u.$
\end{remark}

Assuming that the weights are bounded below by a positive constant, and have a
particular logarithmic modulus of continuity, we are able to show that
$d\mu_{\bb}$ is a doubling measure and the Dirichlet form satisfies a scale
invariant $L^2$-Poincar\'e inequality. The Sobolev inequality and then the
Harnack inequality follow from an argument of Saloff-Coste, which K. Sturm
adapted to the metric-measure category.  These are local estimates that can
then be applied to solutions defined on a compact manifold with corners.
As an important consequence we then show that the regular solution to the
Cauchy problem with initial data in $\cC^0$ is H\"older continuous for positive
times. 

A generalized Kimura diffusion operator, $L$ on a compact manifold with
corners, $P,$ defines a measure $d\mu_{L},$ given locally by
$e^{U(x;y)}d\mu_{\bb},$ as in~\eqref{eqn3.004}. This is a finite measure if the
weights are strictly positive. The following basic regularity result is a
consequence of our local estimates:
\begin{theorem}\label{thm5.1.0011}[Theorem~\ref{thm5.1.001}]
  Let $P$ be a compact manifold with corners and $L$ a generalized Kimura diffusion
  operator with smooth coefficients defined on $P.$ Suppose that the weights
  defined by $L$ are positive  along every boundary
  component. If $u$ is a weak solution to the initial value problem
  \begin{equation}
    (\pa_t-L)u=0\text{ with }u(\xi,0)=f(\xi)\in L^2(P;d\mu_{L}),
  \end{equation}
then $u\in\CI(P\times (0,\infty)).$
\end{theorem}
\noindent
Among other things we also show that $(\mu-L)^{-1}$ acting on $\cC^0(P)$ is a
compact operator.  In addition we establish upper bounds for the ``heat
kernel,'' i.e. the Schwartz kernel of $e^{tL}.$  Our earlier work indicates that
this heat kernel is smooth along the boundary in the outgoing variables, but
somewhat singular along the boundary in the incoming variables.

  Let $\rho_i(\xi,\eta)$ denote the distance between $\xi, \eta\in P,$ with respect to the incomplete metric defined by dualizing the principal symbol of $P.$
\begin{theorem}\label{cor5.2.22}[Theorem\ref{cor5.2.2}] Assume that $P$ is a compact manifold with
  corners and $L$ is a generalized Kimura diffusion defined on $P$ with
  positive weights. If we represent the kernel of the operator $e^{tL}$ as
  $p_t(\xi,\eta)d\mu_{L}(\eta),$ then  there  are positive constants $C_0,C_1,C_2$
  so that, for all $t>0$ and pairs $\xi,\eta\in P$ we have
\begin{equation}\label{eqn207.00551}
 p_t(\xi,\eta)\leq \frac{C_0\exp\left(-\frac{\rho_i^2(\xi,\eta)}{C_2t}\right)}
{\sqrt{\mu_{L}(B^i_{\sqrt{t}}(\xi))\mu_{L}(B^i_{\sqrt{t}}(\eta))}}
\times\left(1+\frac{\rho_i(\xi,\eta)}{\sqrt{t}}\right)^{D}
\cdot\exp(C_1t).
\end{equation}
For each $\eta\in P,$ the function $(\xi,t)\mapsto p_t(\xi,\eta)$ belongs
to $\CI(P\times (0,\infty)).$
\end{theorem}
\noindent
In particular, $p_t(\xi,\eta)$ is bounded for positive times, which shows that
the leading singularity of the heat kernel on the incoming face is captured by
the measure $d\mu_{L}.$ In~\cite{shimakura1} Shimakura gives a similar estimate
for the heat kernel of the standard Kimura diffusion operator on the simplex in
$\bbR^d,$ under the assumption that the weights are constant and at least
$1/2.$ In~\cite{ChenStroock} Chen and Stroock prove an analogous result in the
1-dimensional case, with vanishing weights.

In a separate paper we  treat a special subclass of ``diagonal operators,'' 
which act on functions defined on $S_{n,m}$ and take the special form
\begin{equation}
  Lu=\sum_{j=1}^n[x_i\pa_{x_i}^2+b_i(x;y)\pa_{x_i}]u+\sum_{l=1}^m[\pa_{y_l}^2+d_l(x;y)\pa_{y_l}]u.
\end{equation}
We analyze this special case using the kernel methods introduced in~\cite{WF1d}
and~\cite{EpMaz2}. Assuming that the weights are bounded below, and 
that the coefficients, $\{b_i,d_l\},$ are constant outside a compact set, we
establish the H\"older regularity of solutions to $(\pa_t-L)u=0$ with initial
data in $\cC^0_c(S_{n,m}).$ 

\smallskip \centerline{\bf Acknowledgements} We would like to thank Daniel
Stroock for suggesting that we look at the works of Saloff-Coste, Grigor'yan
and Sturm. We would like to thank Camelia Pop for many helpful discussions and
her many independent contributions to this effort; we would also like to thank
Phil Gressman for many useful conversations over the course of this research.

\section{A Preliminary Result}\label{sec2}

The weights, which are the coefficients of the transverse components of the
vector field along $\pa P,$ play a central role in this paper. We first show
that they are invariantly defined by the operator itself.  Let $P$ be a
manifold with corners, and $\xi\in \pa P$ a boundary point of codimension $n.$ If
$L$ is a generalized Kimura diffusion operator, then Proposition 2.2.3
in~\cite{EpMaz2} shows that there are adapted local coordinates
$(x_1,\dots,x_n;y_1,\dots,y_m)$ lying in $S_{n,m}$ with $p\leftrightarrow
(0;0),$ in which $L$ takes the form
\begin{multline}\label{eqn1.001}
  Lu=\sum_{j=1}^n[x_i\pa_{x_i}^2+b_i(x;y)\pa_{x_i}]u+\sum_{i,j=1}^n+
x_ix_ja_{ij}(x;y)\pa_{x_i}\pa_{x_j}u+\\
\sum_{i=1}^n\sum_{l=1}^mx_ic_{il}\pa_{x_i}\pa_{y_l}u+
\sum_{k,l=1}^md_{kl}(x;y)\pa_{y_k}\pa_{y_l}u+\sum_{l=1}^md_l(x;y)\pa_{y_l}u.
\end{multline}
The operator is assumed to be elliptic where $\{x_i>0:\: i=1,\dots,n\},$ and
coefficients of the transverse vector field $\{b_i(x;y):\:i=1,\dots,n\}$ are
non-negative along the boundary, i.e. $b_i(x;y)\geq 0,$ where $x_i=0.$ A given
point can belong to a variety of such coordinate charts, nonetheless, as shown
below, these coefficients are invariantly defined.

Label the hypersurface boundary components of $P$ by indices $\cI:$
\begin{equation}
  \pa P=\bigcup_{i\in\cI}H_i.
\end{equation}
To demonstrate this invariance, recall that the principal symbol of the
operator $L$ in the interior of $P$ is a positive definite quadratic form
on the fibers of $T^*P;$ by duality, it defines an incomplete metric on $P.$ Let
$r_i(\eta)$ denote the minimal distance from a point $\eta\in P$ to the boundary
hypersurface with index $i.$ Each $r_i$ is smooth in a neighborhood of
$\overline{H_i}\subset P.$ Suppose that the point $\xi\in \pa P$ is of
codimension $n$ and $(x_1,\dots,x_n;y_1,\dots,y_m)$ are adapted local
coordinates centered at $\xi.$ There are distinct indices $\{i_1,\dots,i_n\}$ so
that
\begin{equation}
  \xi\in \bigcap_{j=1}^nH_{i_j}.
\end{equation}
Moreover, upon relabeling we recall from the construction of adapted local
coordinates that 
\begin{equation}\label{eqn18.013}
  2\sqrt{x_i}=r_{j_i}(x;y).
\end{equation}
In these coordinates, the operator takes the form~\eqref{eqn1.001} from which it 
is clear that, for each $i:$
\begin{equation}
  b_i(x;y)=Lx_i\restrictedto_{x_i=0}=\frac{1}{4}Lr^2_{j_i}\restrictedto_{r_{j_i}=0}.
\end{equation}
The last expression is globally defined along $H_i,$  completing the proof
of the following proposition:
\begin{proposition}\label{prop1.001}
  Let $P$ be a manifold with corners and $L$ a generalized Kimura diffusion
  operator defined on $P.$ The coefficients of the transverse vector field
  along the boundary of $P,$ in any adapted coordinate system, are restrictions
  of functions defined globally on the hypersurfaces.
\end{proposition}

These functions are of central importance in what follows, so we make the following definition:
\begin{definition} The normalized coefficients of the transverse vector fields along $\pa P$
  defined by a generalized Kimura diffusion operator,
  $\{\frac{1}{4}Lr^2_i\restrictedto_{r_{i}=0}:\:i\in\cI\}$  are called the
  \emph{weights} of the Kimura operator.
\end{definition} 
\noindent
Strictly speaking, the weights are invariantly defined only along $\pa P,$ but we sometimes 
use the term to refer to the functions $\{b_i(x;y)\}$ defined in a 
neighborhood of a subset of $\pa P,$ and which agree with the weights on $\pa P.$ As
described earlier, these weights define a class of measures on $P$, the elements of which 
differ by a bounded, non-vanishing factor of the form $e^{U(\xi)}.$ 

Let $dV_P$ be a smooth non-degenerate density on $P.$ For
each $i\in\cI,$ we let $B_i$ be a smooth extension of the weight $b_i$ from
$H_i$ to all of $P.$ For simplicity, assume that $B_i$ is independent of
$x_i$ in a small neighborhood of $H_i$ and reduces to a positive constant
outside of a slightly larger neighborhood; similarly, let $R_i$ denote a smooth
extension to $P$ of $r_i,$ the distance to $H_i,$ which we again assume is a
positive constant outside of a small neighborhood of $H_i.$ Set
\begin{equation}
  W(\xi)=\prod_{i\in\cI} R_i(\xi)^{2B_i(\xi)-1},
\end{equation}
and define the measure $d\mu_L$ by
\begin{equation}
  d\mu_L(\xi)=W(\xi)dV_P(\xi).
\end{equation}
It follows from~\eqref{eqn18.013} that in any adapted coordinate system, $(\bx;\by),$ there is a 
bounded, continuous function $U(\bx;\by)$ so that $d\mu_{L}(\bx;\by)=e^{U(\bx;\by)}d\mu_{\bb}(\bx;\by).$ 
The expansion of $U$ along $H_i$ typically takes the form $\gamma(\bx;\by)x_i\log x_i,$ so it is not in 
general smooth. We speak of a ``measure defined by the weights of $L$'' as any measure with this property. 

\section{Metric-measure estimates}\label{sec3}
We now turn to the analysis of the class of generalized Kimura diffusion
operators that can be locally defined using a symmetric Dirichlet form, as
in~\eqref{eqn1.009}.  Elements of this class of Kimura diffusion operators are 
both more general than the ones considered before, as certain coefficients 
of the first order terms are allowed to be singular, but also less general 
in that not every Kimura diffusion operator has such a description, even locally.  
This approach to proving estimates is an outgrowth of the pioneering work of John Nash and J\"urgen
Moser on estimates for elliptic and parabolic equations with bounded measurable
coefficients. More recently these ideas have been recast by Fabes and Stroock
on the one hand, and Davies, Saloff-Coste, and Grigor'yan on the other, as a
way to obtain Harnack inequalities, H\"older estimates on solutions and
kernel bounds for the Green and heat kernels defined by uniformly elliptic operators.  

Briefly, this approach uses Moser's iteration to obtain bounds on solutions to
elliptic and parabolic equations via the Sobolev inequality and properties of
doubling measures. For uniformly elliptic operators on manifolds, Saloff-Coste
and Grigor'yan isolated the two essential ingredients: that the measure have
the doubling property and that there is a scale-invariant $L^2$-Poincar\'e
inequality.  This was later generalized by Sturm, see~\cite{Sturm1,Sturm2,
  Sturm3}, to the setting of metric measure spaces, with operators defined
through strongly local Dirichlet forms. Sturm's work provided a strong impetus
to adopt this general approach, but in the end, we often found it easier to adapt the
proofs given in~\cite{SaloffCosteLMS}, rather than to use Sturm's results
directly.

The underlying space is $S_{n,m},$ which is a manifold with corners, and should be understood 
as an adapted coordinate chart for $P$. We endow this chart with the measure 
\begin{equation}\label{eqn66}
d\mu_{\bb}(x;y)=\prod_{j=1}^nx_i^{b_i(x;y)-1}dxdy, 
\end{equation}
where 
$$\bb=(b_1(x;y).\dots,b_n(x;y))$$ 
is a vector of positive continuous functions, which are constant outside a bounded neighborhood of $(0;0).$ 
Fix a constant $\beta_0>0$ such that 
\begin{equation}
  \beta_0\leq b_i(x;y)\text{ for all }(x;y)\in S_{n,m}.
\end{equation}
For many applications it is reasonable, even necessary, to assume that these functions are
$\cC^1$ in the variables $(\bx;\by),$ or else in the ``square-root'' variables $(\sqrt{\bx};\by)$;
however, many of the basic results below require far less regularity.

If $B \subset S_{n,m}$ is relatively open and $u, v \in \CIc(B)$, then we consider the Dirichlet form
\begin{multline}\label{eqn68}
  Q_B(u,v)=\int\limits_{B}\Bigg[\sum_{j=1}^nx_j\pa_{x_j}u\pa_{x_j}v+\sum_{i,j}\sqrt{x_ix_j}
a_{ij}(x;y)\pa_{x_i}u\pa_{x_j}v+\\
\frac{1}{2}\sum_{j,l}\sqrt{x_j}c_{jl}(x;y)[\pa_{x_j}u\pa_{y_l}v+\pa_{x_j}v\pa_{y_l}u]+
\sum_{l,m}d_{lm}(x;y)\pa_{y_l}u\pa_{y_m}v \Bigg]d\mu_{\bb}(x;y),
\end{multline}
Note that $u, v$ are \emph{not} required to vanish along $\pa S_{n,m}\cap \overline{B}.$ 

We define the associated $L^2$-inner product by setting:
\begin{equation}
(u,v)_{\bb,B}=\int\limits_{B}uv \, d\mu_{\bb}(x;y).
\end{equation}
The subscript $B$ is omitted if the intended subset is clear from the context. 
Formally integrating by parts, assuming for example that $v$ vanishes near $\pa S_{n,m}$, gives
\begin{equation}
Q_B(u,v)=-(L_Qu,v)_{\bb,B},
\end{equation}
where
\begin{multline}
L_Qu=\sum_{j=1}^n[x_j\pa^2_{x_j}+b_j(x;y)\pa_{x_j}]u+\sum_{i,j}\pa_{x_j}\sqrt{x_ix_j}
a_{ij}(x;y)\pa_{x_i}u+\\
\frac{1}{2}\sum_{j,l}[\pa_{y_l}\sqrt{x_j}c_{jl}(x;y)\pa_{x_j}u+\pa_{x_j}\sqrt{x_j}c_{jl}(x;y)\pa_{y_l}u]+
\sum_{l,m}d_{lm}(x;y)\pa_{y_m}\pa_{y_l}u+Vu.
\end{multline}
The vector field $V$ is tangent to $\pa S_{n,m}$, but note that if $\bb(x;y)$ is non-constant near $\pa
S_{n,m}$, then $V$ may have singular coefficients and involve terms of the form
\begin{equation}\label{eqn19.006}
 \log x_j(x_i \pa_{x_i}),\quad \log x_j\, \pa_{y_l}.
\end{equation}
All of this works equally well on more general manifolds with corners. This
setup is related to the ideas used by Shimakura in~\cite{shimakura1} to study
certain simple Kimura-type operators defined on simplices.  In Shimakura's work
the weights are assumed to be constant.

If $L$ is a generalized Kimura diffusion operator on $P,$ then $L$ determines a
class of measures, as noted above, which can be taken to have the form
$d\mu_{\bb}$ in an adapted coordinate chart. This measure and the principal
symbol of $L$ then determine the symmetric quadratic form $Q$. Conversely, if
$L_Q$ is the second order operator determined from $Q$ as above, then in a
neighborhood of a point on $\pa P,$ the difference $L-L_Q$ is a vector field
tangent to $\pa P,$ possibly with mildly singular coefficients, as
in~\eqref{eqn19.006}. The estimates produced by the Moser method are local,
which allows us to establish H\"older regularity for weak solutions of many
classes of generalized Kimura diffusion operators.

We assume that symmetric quadratic form is positive definite in the interior of $S_{n,m}$. Because
of the form of the coefficients, there is a naturally induced quadratic form on any boundary stratum,
and we assume that each of these is also positive definite on the interior of that stratum.  We assume
finally that the coefficients of the quadratic form, 
$$\{a_{ij}(x;y),c_{jl}(x;y),d_{lm}(x;y)\},$$
are smooth functions of the variables  $(\sqrt{\bx};\by).$ Integrating by parts shows that the
natural boundary conditions are the same as those defining the ``regular solution'' for a generalized 
Kimura operator introduced in~\cite{EpMaz2}, see Section~\ref{ssec4.1.006}. 

Writing the integrand symbolically as $\langle A(x;y)\nabla u,\nabla v\rangle,$ we also consider 
operators which include non-symmetric terms of the form
\begin{equation}
  \int\limits_{B}\langle A(x;y)\nabla u,X(x;y)\rangle \, d\mu_{\bb}(x;y),
\end{equation}
where $X(x;y)=(X_1(x;y),\dots,X_{n+m}(x;y))$ are continuous in the interior of $S_{n,m}.$ 
That is, we allow the addition of an arbitrary continuous tangent vector field. We defer the 
development of this case to Section~\ref{sec4.2.001}, and focus on the symmetric case.

Changing variables in the Dirichlet form by $w_i=\sqrt{x_i}$ gives
\begin{equation}
d\mu_{\bb}(w;y)=2^n\prod_{i=1}^nw_i^{2b_i(w;y)-1}dwdy,
\end{equation}
and 
\begin{multline}\label{eqn23.009}
Q_B(u,v)=\frac{1}{4}\int\limits_{B}\Bigg[\sum_{j=1}^n\pa_{w_j}u\pa_{w_j}v+\sum_{i,j}
a_{ij}\pa_{w_i}u\pa_{w_j}v+\\
2\sum_{j,l}c_{jl}[\pa_{w_j}u\pa_{y_l}v+\pa_{w_j}v\pa_{y_l}u]+
4\sum_{l,m}d_{lm}\pa_{y_l}u\pa_{y_m}v \Bigg]d\mu_{\bb}(w;y).
\end{multline}
The ellipticity hypothesis is that 
\begin{equation}\label{eqn88.006}
q_{(w;y)}(\xi,\eta)=\frac{1}{4} \Bigg[\sum_{j=1}^n\xi_j^2+\sum_{i,j} 
a_{ij}\xi_i\xi_j+ 4\sum_{j,l}c_{jl}\xi_j\eta_l+ 4\sum_{l,m}d_{lm}\eta_l\eta_m\Bigg]
\end{equation}
(the coefficients of which are constant outside a compact set) is positive definite, i.e., 
there are positive constants $\lambda, \Lambda$ so that
\begin{equation}\label{eqn74}
  \lambda(|\xi|^2+|\eta|^2)\leq q_{(w;y)}(\xi,\eta)\leq \Lambda(|\xi|^2+|\eta|^2).
\end{equation}
Since we are primarily interested in local estimates near the corner $(0;0),$ we assume that
\begin{equation}\label{eqn90.006}
  \text{The quadratic form in~\eqref{eqn88.006} is uniformly elliptic in $S_{n,m}.$}
\end{equation}
Observe that this is invariant under the dilations $(w;y)=(\mu w';\mu y'),$ $\mu>0,$ which transform the 
measure $d\mu_{\bb}(w;y)$ to 
\begin{equation}
  \mu^{m}2^n\prod_{i=1}^n(w_i')^{2\tb_i(w';y')-1}e^{\phi(w',y')\log\mu}dw'dy',
\end{equation}
where 
\begin{equation}
\tb_i(w',y')=b_i(\mu w',\mu y'), \quad \mbox{and} \quad \phi(w',y')=2\sum_{i=1}^nb_i(\mu w',\mu y').
\end{equation}

Sturm introduces the notation that if $u$ is in the domain of $Q$, then there is a measure 
$d\Gamma(u,u)$ so that
\begin{equation}
Q(u,u)=\int\limits_{S_{n,m}}d\Gamma(u,u).
\end{equation}
In our case 
\begin{equation}
d\Gamma(u,u)(w;y)=\langle A(w;y)\nabla u,\nabla u\rangle d\mu_{\bb}(w;y).
\end{equation}
Next, in terms of the space of functions
\begin{equation}
\cU_1=\{u:\:\langle A(w;y)\nabla u,\nabla u\rangle\leq 1\},
\end{equation}
the \emph{intrinsic metric} is defined by 
\begin{equation}
\rho_i((w_1;y_1),(w_2;y_2))=\sup\{u(w_1;y_1)-u(w_2;y_2):\: u\in \cU_1\}.
\end{equation}
By~\eqref{eqn90.006}, this intrinsic metric is uniformly equivalent to the Euclidean metric,
\begin{equation}
\rho^e_2((w_1;y_1),(w_2;y_2))= \left(\|w_1-w_2\|^2_2+\|y_1-y_2\|^2_2\right)^{\frac12},
\end{equation}
or equivalently, in terms of the $(x;y)$ coordinates, 
\begin{equation}
\rho^e_2((x_1;y_1),(x_2;y_2))=\left(\sum_{j=1}^n|\sqrt{x_{1j}}-\sqrt{x_{2j}}|^2+\|y_1-y_2\|^2_2\right)^{\frac12}.
\end{equation}
This determines the standard topology on $S_{n,m}$. It is equivalent to the
metric used in~\cite{EpMaz2} to define the anisotropic H\"older spaces
$\cC^{k,\gamma}_{\WF}$ and $\cC^{k,2+\gamma}_{\WF},$ which play a key role in
the analysis of generalized Kimura diffusion operators. 
The ball of radius $r$ centered at $(w;y)$ with respect to $\rho^e_{2}$ is
denoted $B^e_r(w;y)$.

The main estimates on the heat kernel and solutions to the heat equations  follow by a rather general
argument once we prove that: 
\begin{enumerate}
\item  The measure, $d\mu_{\bb}$ is a doubling measure, and
\item The Dirichlet forms, $Q_{B^i_r}$ satisfy scale-free $L^2$ Poincar\'e
  inequalities, for intrinsic-metric balls $B^i_r$.
\end{enumerate}
The proofs of these facts both proceed by checking their validity when $\bb$ is
constant and then using perturbative arguments to conclude their validity in general.
The details of this analysis occupy the remainder of this section.

First observe that for the purposes of proving the Poincar\'e inequality, we
may replace the quadratic form $q(\nabla_wu,\nabla_yu),$ defined
in~\eqref{eqn88.006}, with the standard Euclidean one, giving the equivalent
Dirichlet form
\begin{equation}
Q^e_{B}(u,u)=\int\limits_{B}\left[\sum_{j=1}^n|\pa_{w_j}u(\tw;\ty)|^2+
\sum_{l=1}^m|\pa_{y_l}u(\tw;\ty)|^2\right]d\mu_{\bb}(\tw;\ty),
\end{equation}
while of course retaining the same measure $d\mu_{\bb}$. We may also use the equivalent $\ell^\infty$ metric,
\begin{equation}
\rho^e_{\infty}((w_1;y_1),(w_2;y_2))=\max\{\|w_1-w_2\|_{\infty},\|y_1-y_2\|_{\infty}\},
\end{equation}
with respect to which closed balls now have the form
\begin{equation}
B_r(w;y)=\prod_{j=1}^n[\max\{w_i-r,0\},w_i+r]\times \prod_{l=1}^m[y_l-r,y_l+r].
\end{equation}
Indeed, the inclusions 
$$B_r^e(w;y)\subset B_r(w;y)\subset B^e_{\sqrt{n+m}\cdot r}(w;y)$$
show that $d\mu_{\bb}$ is a doubling measure with respect to one set of balls if and
only if it is a doubling measure with respect to the other. As for the Poincar\'e inequality, 
suppose that we prove that there exists a constant $C > 0$ such that
\begin{equation}\label{eqn79.004}
\int\limits_{B_r(w;y)}|u(\tw;\ty)-u_{B_r(w;y)}|^2d\mu_{\bb}(\tw;\ty)\leq C r^2Q_{B_r(w;y)}(u,u)
\end{equation}
for all $r > 0$ and $(w;y)\in S_{n,m}$, where (for any measurable set $B$), $u_B$ is the average 
\begin{equation}
u_{B}=\frac{1}{\mu_{\bb}(B)}\int\limits_{B}u(\tw;\ty)d\mu_{\bb}(\tw;\ty), \qquad 
\mu_{\bb}(B)=\int\limits_{B} d\mu_{\bb}(\tw;\ty).
\end{equation}
We then have that 
\begin{equation}
\begin{split}
\int\limits_{B^e_r(w;y)}|u(\tw;\ty)-u_{B^e_r(w;y)}|^2d\mu_{\bb}(\tw;\ty)&\leq
  \int\limits_{B^e_r(w;y)}|u(\tw;\ty)-u_{B_r(w;y)}|^2d\mu_{\bb}(\tw;\ty)\\
&\leq\frac{C}{n+m} [(n+m)r^2]Q_{B^e_{r\sqrt{n+m}}(w;y)}(u,u).
\end{split}
\end{equation}
In other words, the strong Poincar\'e inequality for the family of $\{B_r(w;y)\}$ implies a weak Poincar\'e inequality 
for the balls $\{B^e_r(w;y)\}.$ Theorem 2.4 from~\cite{Sturm3} then implies that the strong Poincar\'e also holds for 
the balls $B^e_r(w;y)$.  We obtain the estimate~\eqref{eqn79.004} following a well-known argument of Jerison~\cite{Jerison1}, 
who shows how to pass from a weak scale-invariant Poincar\'e inequality to a
strong one. 

There is one further preparatory remark. The volume doubling and Poincar\'e
inequality, and hence the various conclusions that they imply, require very
little regularity for the functions $\{b_i\}$. The minimal condition that naturally
emerges here is that there is a constant $C$ so that for each $i$,
\begin{equation}
|b_i(w; y) - b_i(\tw; \ty)| \leq \frac{C}{|\log \rho_\infty^e( (w;y) , (\tw; \ty) )|}.  
\label{regb}
\end{equation}

\bigskip

We first prove that $d\mu_{\bb}$ is a doubling measure:
\begin{proposition}\label{prop3.1.0004} 
Let $\{b_i(w;y)\}$ be positive functions which are constant outside a compact set and
satisfy \eqref{regb}. Then there is a constant $D$ so that for any $r>0$ and $(w;y)\in S_{n,m}$, 
\begin{equation}\label{eqn78.0003}
\mu_{\bb}(B_{2r}(w;y)) \leq 2^D \mu_{\bb}(B_{r}(w;y)).
\end{equation}
\end{proposition}
\begin{proof} 
We first verify this when the $b_i$ are everywhere constant. Since each $B_r(w;y)$ is a product of
intervals, we immediately reduce to the one-dimensional case, where $B_r(w_i) = (\max\{w_i-r,0\}, w_i + r)$,
and hence for the measure $\mu_b = w_i^{2b-1}dw_i$, 
\begin{equation}
\mu_b(B_r(w_i)) = 
\begin{cases} \frac{(w_i+r)^{2b}}{2b}&\text{ if  }w_i\leq r \\ 
\frac{(w_i+r)^{2b}-(w_i-r)^{2b}}{2b}&\text{ if }w_i> r. 
\end{cases}
\end{equation}
It follows directly from this that for some constant $C_b > 0$, 
\begin{equation}
\frac{1}{C_b}r^{2b}\leq \mu_{b}(B_r(w_i))\leq C_b r^{2b},\text{ if }w_i\leq 4r,
\end{equation}
and 
\begin{equation}
\frac{1}{C_b}w_i^{2b-1}r\leq \mu_b(B_r(w_i))\leq C_b w_i^{2b-1}r\text{ if }w_i>r.
\end{equation}
The doubling inequality \eqref{eqn78.0003} follows immediately from these
estimates in this case.

For the general case, we need to show that the quotient
\[
F(w;y, r) := \frac{\mu_{\bb}(B_{2r}(w;y))}{\mu_{\bb}(B_{r}(w;y))}
\]
is uniformly bounded from above.  Suppose that it is not, i.e., there exists a
sequence of radii $r_\ell$ and centers of balls $(w^{(\ell)}; y^{(\ell)})$ such
that $F(w^{(\ell)}; y^{(\ell)}, r_\ell)$ tends to infinity. Since $F$ is
clearly continuous in its arguments $(w;y) \in S_{n,m}$ and $0 < r < \infty$,
this unboundedness could only occur if either $(w^{(\ell)}; y^{(\ell)})$
diverges (to infinity or $\pa P$) or else $r_\ell$ tends to $0,$ or $\infty.$ We
shall rule these possibilities out in turn.

The first case, where $r_\ell \nearrow \infty$, is easy. Suppose that the
functions $\{b_i\}$ are constant outside the ball $B_R(0;0)$.  Consider the worst
case, when $(w^{(\ell)};y^{(\ell)}) = (0;0)$. But then, for $\rho \gg R$,
\[
\mu_{\bb}(B_\rho(0;0)) = \mu_{\bb^0}(B_\rho(0;0)) + A, \quad
A = \mu_{\bb}(B_R(0;0)) - \mu_{\bb^0}(B_R(0;0)),
\]
where $\bb^0$ is the constant value of $\bb$ outside a compact set. The uniform
upper bound for $F(0;0; r_\ell)$ is then straightforward. A slightly more
complicated estimate, which we leave to the reader, is required when the center
of the ball does not lie at the origin, but the same conclusion still holds.
The case where $r_\ell$ remains in a bounded interval $0 < \underline{r} \leq
r_\ell \leq \overline{r} < \infty$, but $(w^{(\ell)};y^{(\ell)}) \to \infty$,
is covered by the computations when $\bb$ is constant.

Finally, suppose that the centers $(w^{(\ell)};y^{(\ell)})$ remain in $B_{2R}(0;0)$ and $r_\ell \searrow 0$. 
Change variables, setting $w_i = r_\ell \hat{w}_i$, $y_i = r_\ell \hat{y}_i$, where $(\hat{w}; \hat{y}) \in
B_{2R/r_\ell}(0;0)$. The centers $(w^{(\ell)};y^{(\ell)})$ are transformed to new points $(\hat{w}^{(\ell)};\hat{y}^{(\ell)})$, and
\[
F(w^{(\ell)};y^{(\ell)}, r_\ell) = \frac{\int_{B_2(\hat{w}^{(\ell)}; \hat{y}^{(\ell)})} \prod \hat{w}_i^{2 \hat{b}_i(\hat{w}; \hat{y})-1} 
e^{(2\hat{b}_i(\hat{w};\hat{y})-1)\log r_\ell}\, d\hat{w} d\hat{y}}{\int_{B_1(\hat{w}^{(\ell)}; \hat{y}^{(\ell)})} \prod \hat{w}_i^{2 \hat{b}_i(\hat{w}; \hat{y})-1} 
e^{(2\hat{b}_i(\hat{w};\hat{y})-1)\log r_\ell}\, d\hat{w} d\hat{y}}. 
\]
Here $b_i(\hat{w}; \hat{y}) = b_i(r \hat{w}; r\hat{y}) = b_i(w;y)$. Note that a common factor of $r_\ell^{m+n}$ has been 
cancelled from both the numerator and denominator.  The dependence on $r_\ell$ is now entirely contained in the
functions $\hat{\bb}(\hat{w}; \hat{y})$. We are aided by the fact that each
$\hat{b}_i$ takes values in some interval $0< \underline{\beta} \leq
\hat{b}_i \leq \overline{\beta} < \infty$. Now substitute 
\[
\hat{b}_i(\hat{w};\hat{y}) = \hat{b}_i(\hat{w}^{(\ell)};\hat{y}^{(\ell)}) + \beta_i(r_\ell(\hat{w};\hat{y}), r_\ell(\hat{w}^{(\ell)};\hat{y}^{(\ell)}))
\]
into the final exponent in each integrand.  The expression 
\[
\exp((2\hat{b}_i(\hat{w}^{(\ell)};\hat{y}^{(\ell)})-1)\log r_\ell)
\]
is constant and appears in both the numerator and denominator, hence may be cancelled.  We are left with
\[
\frac{\int_{B_2(\hat{w}^{(\ell)}; \hat{y}^{(\ell)})} \prod \hat{w}_i^{2 \hat{b}_i(\hat{w}; \hat{y})-1} e^{\beta_i \log r_\ell}\, d\hat{w} 
d\hat{y}}{\int_{B_1(\hat{w}^{(\ell)}; \hat{y}^{(\ell)})} \prod \hat{w}_i^{2 \hat{b}_i(\hat{w}; \hat{y})-1} 
e^{\beta_i\log r_\ell}\, d\hat{w} d\hat{y}},  
\]
where we omit the arguments of the $\beta_i$ for simplicity. According to \eqref{regb}, 
\begin{equation}
|\beta_i(r_\ell(\hat{w};\hat{y}), r_\ell(\hat{w}^{(\ell)};\hat{y}^{(\ell)}))|
\, |\log r_\ell| \leq \frac{ C |\log r_\ell|}{| \log r_{\ell}+
\log \rho_\infty^e((\hat{w};\hat{y}), (\hat{w}^{(\ell)};\hat{y}^{(\ell)}))|}
\leq C'
\end{equation}
since $\rho_\infty^e( (\hat{w};\hat{y}), (\hat{w}^{(\ell)}; \hat{y}^{(\ell)})) \leq 2$.   Hence these second factors are
bounded above and below, and may be disregarded. 

There are now two final cases to analyze. In the first, the centers
$(\hat{w}^{(\ell)}; \hat{y}^{(\ell)})$ remain bounded (relative to the
$(\hat{w}; \hat{y})$ coordinate system), and in the second they do not.  The
first case is slightly easier, since we may assume that $<(\hat{w}^{(\ell)};
\hat{y}^{(\ell)})>$ converges, and then simply pass to the limit $r_\ell \to
0.$ Both the numerator and denominator have finite, positive limits, and so we
conclude that this sequence of quotients is bounded after all.  In the second
case, the numerator and denominator each tend to infinity with $\ell$.  The
functions $\hat{\bb}$ are constant outside the ball $B_{R/r_\ell}(0;0)$. If the
centers $(\hat{w}^{(\ell)}; \hat{y}^{(\ell)})$ lie outside this ball, then the
quotient is clearly bounded. So the only remaining case is when
$(\hat{w}^{(\ell)}; \hat{y}^{(\ell)}) \in B_{R/r_\ell}(0;0)$.  We may now
perform the same substitution as above, writing each $\hat{w}_i^{2
  \hat{b}_i(\hat{w}; \hat{y}) - 1}$ as $\hat{w}_i^{2
  \hat{b}_i(\hat{w}^{(\ell)}; \hat{y}^{(\ell)}) - 1} e^{ \beta_i \log
  \hat{w}_i}$ Since $\log \hat{w}_i \leq C(1 + \log r_\ell)$, we may apply
exactly the same reasoning as above to neglect these error terms, and then
cancel the remaining constant terms.

This contradiction demonstrates that the quotient is uniformly bounded as $r
\searrow 0,$ and therefore that $\mu_{\bb}$ is a doubling measure.
\end{proof}
\begin{remark} This Lemma is slightly more complicated than one might expect because the measures $\mu_{\bb}$ 
are not Ahlfors $(m+n)$-regular. Indeed, $\mu_{\bb}(B_r)$ is bounded above and below by constant multiples of 
$r^{m+n}$ provided the ball does not intersect the boundary, but these constants are not uniform. The decay rates
of measures of small balls centered at any boundary point are given by different powers of $r.$ Hence 
our problem provides an interesting example where the most common version of Moser's arguments to get heat 
kernel bounds does not apply, since these require Ahlfors regularity, so the variant of these arguments
given in~\cite{SaloffCosteLMS} is needed. 
\end{remark}

We now turn to the proof of the scale-invariant Poincar\'e inequality. As we have explained earlier, 
it suffices to prove the following result:
\begin{theorem}\label{prop3.2.003} Suppose that the functions $\{b_i(w;y):\:i=1,\dots,n\},$
defined in $\bbR_+^n\times\bbR^m,$ satisfy \eqref{regb}, are bounded below by a positive constant, and 
are constant outside a compact set. Then there is a constant $C$ so that for any $0<r,$ and $(w;y)\in S_{n,m}$ and
$u\in\cC^1(\overline{B^i_r(w;y)})$, 
\begin{equation}\label{eqn179.005}
\int\limits_{B^i_r(w;y)}|u(\tw;\ty)-u_{B_r(w;y)}|^2d\mu_{\bb}(\tw;\ty)\leq Cr^2Q^e_{B^i_r(w;y)}(u,u).
\end{equation}
\end{theorem}

The proof is somewhat more complicated than in the doubling measure result above; it uses
a covering argument due to Jerison~\cite{Jerison1} which produces a cover of $B^i_r(w;y)$ by smaller 
balls where the approximation of the functions $\{b_i(\tw;\ty)\}$ by constants is permissible.  
Thus the first step is to prove the result when the $b_i$ are all constant for
balls with respect to the $\rho^e_{\infty}$-metric. 

\begin{proposition}\label{prop3.2.002}
  Let $\bb_0=(b_{01},\dots,b_{0n})$ be a vector of positive constants. There is
  a constant $C_{\bb_0}$ so that for all $0<r,$ $(w;y)\in S_{n,m}$ and
  $u\in\cC^1(\overline{B_r(w;y)})$ we have the estimate
\begin{equation}\label{eqn46.009}
\int\limits_{B_r(w;y)}|u-u_{\bb_0}|^2d\mu_{\bb_0}(\tw,\ty)\leq C_{\bb_0}r^2\int\limits_{B_r(w;y)}
\left| \nabla_{\tw; \ty} u \right|^2 d\mu_{\bb_0}(\tw,\ty),
\end{equation}
where
\begin{equation}
u_{\bb_0}=\frac{1}{\mu_{\bb_0}(B_r(w;y))}\int\limits_{B_r(w;y)}ud\mu_{\bb_0}.
\end{equation}
Moreover there is a constant $C_{\beta, B,m}$ so that if $0<\beta<b_{0j}<B$ for
$1\leq j\leq n,$ then
\begin{equation}
  C_{\bb_0}\leq C_{\beta,B,m}.
\end{equation}
\end{proposition}
\begin{proof}[Proof of Proposition~\ref{prop3.2.002}]
The sharp constant $1/C_{\bb_0}$ in~\eqref{eqn46.009} is the first non-zero eigenvalue of the operator $L_{\bb_0}$ 
associated to this Dirichlet form, acting on functions on $B_r(w;y),$ which satisfy appropriate ``Neumann'' boundary
conditions.

The formal operator is given by
\begin{equation}
L_{\bb_0}u=-\left[\sum_{j=1}^n\left(\pa_{w_j}^2+\frac{2b_{0j}-1}{w_j}\pa_{w_j}\right)u+\Delta_yu\right].  
\end{equation}
Recall that the ball is a product
\begin{equation}
  B_r(w;y)=\prod_{j=1}^n[\max\{w_j-r,0\},w_j+r]\times \prod_{l=1}^m[y_l-r,y_l+r].
\end{equation}
Since the form domain is $\CI(\overline{B_r(w;y)})$, we see that on smooth elements in the domain of 
$L_{\bb_0},$ the boundary condition is the standard Neumann one on the ``tangential'' boundary: $\ty_l=y_l\pm r,$ 
i.e., $\pa_{\ty_l}u\restrictedto_{\ty_l=y_l\pm r}=0.$ On the right ends of the intervals, 
\begin{equation}
  \lim_{\tw_j\to (w_j+r)^-}\pa_{w_j}u(\tw;\ty)=0,
\end{equation}
and similarly, on the left ends when $w_j-r>0,$ 
\begin{equation}
  \lim_{\tw_j\to (w_j-r)^+}\pa_{w_j}u(\tw;\ty)=0.
\end{equation}
However, when $w_j<r,$ then the boundary condition at the left endpoint becomes
\begin{equation}
   \lim_{\tw_j\to 0^+}\tw_j^{2b_{0j}-1}\pa_{w_j}u(\tw;\ty)=0.
\end{equation}
The domain of the Friedrichs extension of $L_{\bb_0}$ is denoted $\cD(L_{\bb_0}).$  We are
clearly in a setting where the spectral data for $(L_{\bb_0},\cD(L_{\bb_0}))$ can be determined 
by separation of variables. This is one reason why we replaced the Euclidean balls
by sup-norm balls. The eigenfunctions take the form 
\begin{equation}
  f_1(\tw_1)\cdots f_n(\tw_n)g_1(\ty_1)\cdots g_m(\ty_m),
\end{equation}
where each factor is an eigenfunction of the appropriate boundary value problem
in 1-dimension.  The first non-zero eigenvalue of $(L_{\bb_0},\cD(L_{\bb_0}))$
is then the minimum of the first non-trivial eigenvalues of these $(n+1)$ self
adjoint operators.  For the $y$-variable the first non-trivial eigenvalue for
the Neumann operator on an interval of length $2r$ is $\frac{\pi^2}{4 r^2}.$ This
leaves the 1-dimensional problems in the $w$-variables, which we treat in the following  lemma.
\begin{lemma}\label{lem3.1.006}
  If $0<\beta<B$ then there is a positive constant $\lambda_{\beta,B}$ so that
  for $0\leq x,$ $\beta<b<B$
  and $u\in\cC^1([0\vee(x-r),x+r])$ we have the estimate
  \begin{equation}
    \lambda_{\beta,B}\leq \frac{r^2\int\limits_{\max\{x-r,0\}}^{x+r}|\pa_{w}u(w)|^2w^{2b-1}dw}
{\int\limits_{\max\{x-r,0\}}^{x+r}|u(w)-\bar{u}_b|^2w^{2b-1}dw},
  \end{equation}
where
\begin{equation}
  \bar{u}_b=\frac{\int\limits_{\max\{x-r,0\}}^{x+r}u(w)w^{2b-1}dw}
{\int\limits_{\max\{x-r,0\}}^{x+r}w^{2b-1}dw}.
\end{equation}
\end{lemma}
\begin{proof}
Fix $b>0$. For each $0\leq x$ and $0<r$ we need to estimate the infimum, over functions with 
$w^{2b-1}dw$-mean zero, of the quotient:
\begin{equation}
  \frac{r^2\int\limits_{\max\{x-r,0\}}^{x+r}|\pa_{w}u(w)|^2w^{2b-1}dw}
{\int\limits_{\max\{x-r,0\}}^{x+r}|u(w)|^2w^{2b-1}dw}
\end{equation}
Replacing $x$ by $x/r$, we reduce to the case $r=1$, but still with arbitrary center $x\in [0,\infty).$ Let $\bar{u}$ 
denote the mean of $u.$ As usual, there are three cases:
\begin{enumerate}
\item[Case 1:] If $x<1,$ then to estimate
  \begin{equation}\label{eqn141.006}
  \inf_{\{u:\bar{u}=0\}}  \frac{\int\limits_{0}^{x+1}|\pa_{w}u(w)|^2w^{2b-1}dw}
{\int\limits_{0}^{x+1}|u(w)|^2w^{2b-1}dw}.
  \end{equation}
we find eigenfunctions of the operator
\begin{equation}
\begin{split}
  L_bu=-w^{1-2b}\pa_ww^{2b-1}&\pa_wu\text{ with }\\
\lim_{w\to
    0^+}w^{2b-1}\pa_wu(w)=0&\text{ and }\pa_wu(x+1)=0.
\end{split}
\end{equation}
Solutions of the eigenvalue equation
\begin{equation}
  \pa_w^2u+\frac{2b-1}{w}\pa_wu+\lambda^2u=0
\end{equation}
are in terms of $J$-Bessel functions by
\begin{equation}
w^{1-b}[AJ_{1-b}(\lambda w)+BJ_{b-1}(\lambda w)],
\end{equation}
at least for $b\notin\bbN.$  The boundary condition at $w=0$ implies that $A=0.$  Indeed, if $\nu \notin \bbN$, 
then $J_{\nu}(z)=a_{\nu}z^{\nu}(1+O(z^2))$, so $w^{1-b}J_{1-b}(\lambda w) \sim c w^{2-2b}$, and the boundary condition
at $w=0$ eliminates this term. If $b=1,$ the singular solution has leading term $\log w,$ which is again eliminated
by the boundary condition. Finally, for $b\in\bbN,$ $b > 1$, $w^{1-n}J_{n-1}(\lambda w)$ is the only regular solution. 
Thus, the solution is $(\lambda w)^{1-b}J_{b-1}(\lambda w)$ whenever $b > 0$. This is an entire function which
oscillates infinitely many times as $w \to \infty$. Let $z_{1,b}$ be the smallest positive root of the equation:
\begin{equation}
\pa_{z}[z^{1-b}J_{b-1}(z)]=0,    
\end{equation}
then the smallest non-trivial eigenvalue is:
\begin{equation}
\lambda^2_1=\left(\frac{z_{1,b}}{1+x}\right)^2.
\end{equation}
This gives the infimum of the functional in~\eqref{eqn141.006} for any $x\geq 0,$ which proves
useful in the analysis of the next case. 

We now derive bounds for the constant $z_{1,b}$, depending on the upper and lower bounds of $b$. 
For any $b > 0$, define 
\begin{equation}
  \phi_b(\zeta)=\sum_{k=0}^{\infty}\frac{(-1)^k\zeta^k}{k!\Gamma(k+b)}.
\end{equation}
This is an entire function satisfying the ODE
\begin{equation}
  \zeta\pa_{\zeta}^2\phi_b+b\pa_{\zeta}\phi_b+\phi_b=0
\end{equation}
and the functional equation
\begin{equation}
\pa_{\zeta}\phi_b=-\phi_{b+1}.
\end{equation}
A simple calculation shows that there is a constant $C_b$ so that
\begin{equation}
  z^{1-b}J_{b-1}(z)=C_b\phi_b\left(\frac{z^2}{4}\right)
\end{equation}
Thus $\zeta_{1,b}=z^2_{1,b}/4$ is the smallest positive solution to $\phi_{b+1}(\zeta)=0.$ We can
rewrite $\Gamma(b+1)\phi_{b+1}(\zeta)$ as
\begin{multline}
  \Gamma(b+1)\phi_{b+1}(\zeta)=1-\frac{\zeta}{b+1}+\\
\sum_{k=1}^{\infty}\frac{\zeta^{2k}}{(2k)!(b+1)\cdots (b+2k)}\left[1-\frac{\zeta}{(2k+1)(b+2k+1)}\right],
\end{multline}
from which we see that
\begin{equation}
  (b+1)<\zeta_{1,b}.
\end{equation}
Thus for any $0<\beta<B$
there is a constant $1\leq M_{B}$ so that
\begin{equation}
  4(1+\beta)\leq z^2_{1,b}\leq M_{B}\text{ if }\beta\leq b\leq B.
\end{equation}

\item[Case 2:] If $1<x<2,$ it is simpler to estimate
  \begin{equation}
  \inf_{\{u\in \cC^1[x-1,x+1]\}}  \frac{\int\limits_{x-1}^{x+1}|\pa_{w}u(w)|^2w^{2b-1}dw}
{\int\limits_{x-1}^{x+1}|u(w)-\tu|^2w^{2b-1}dw},
  \end{equation}
from below; here
\begin{equation}
  \tu =\frac{\int\limits_{x-1}^{x+1}u(w)w^{2b-1}dw}{\int\limits_{x-1}^{x+1}w^{2b-1}dw}.
\end{equation}

The analysis in the previous case shows that if $\beta<b<B,$ then for $x \geq 0$, 
\begin{equation}\label{eqn153.005}
  \frac{4(1+\beta)}{(1+x)^2}\int\limits_{0}^{x+1}|u(w)-\bar{u}|^2w^{2b-1}dw\leq
\int\limits_{0}^{x+1}|\pa_{w}u(w)|^2w^{2b-1}dw.
\end{equation}
We define the extension of any $u\in\cC^1([x-1,x+1])$ to 
\begin{equation}
  U(w)=\begin{cases} &u(w)\text{ for }w\in [x-1,x+1]\\
&u(x-1)\text{ for }w\in [0,x-1],
\end{cases}
\end{equation}
which is a function on $[0,x+1]$. This extension is admissible for the inequality
in~\eqref{eqn153.005}, so 
\begin{equation}\label{eqn153.006}
\begin{split}
\int\limits_{x-1}^{x+1}|\pa_{w}u(w)|^2w^{2b-1}dw
=& \int\limits_{0}^{x+1}|\pa_{w}U(w)|^2w^{2b-1}dw\\
\geq&
\frac{4(1+\beta)}{(1+x)^2}\int\limits_{0}^{x+1}|U(w)-\bar{U}|^2w^{2b-1}dw\\
\geq&
 \frac{4(1+\beta)}{(1+x)^2}\int\limits_{x-1}^{x+1}|u(w)-\bar{U}|^2w^{2b-1}dw.
\end{split}
\end{equation}
It is a classical fact that the minimum of
\begin{equation}
  \int\limits_{x-1}^{x+1}|u(w)-a|^2w^{2b-1}dw
\end{equation}
is attained only when $a=\tu,$ and therefore
\begin{equation}
\frac{4(1+\beta)}{(1+x)^2}\leq\frac{  \int\limits_{x-1}^{x+1}|\pa_{w}u(w)|^2w^{2b-1}dw}
  {\int\limits_{x-1}^{x+1}|u(w)-\tu|^2w^{2b-1}dw},
\end{equation}
completing the argument in this case as well.
\item[Case 3:] If $2\leq x,$ then observe that
  \begin{equation}
   \frac{\int\limits_{x-1}^{x+1}|\pa_{w}u(w)|^2w^{2b-1}dw}
{\int\limits_{x-1}^{x+1}|u(w)|^2w^{2b-1}dw}\geq 
\left(\frac{x-1}{x+1}\right)^{|2b-1|}
\frac{\int\limits_{x-1}^{x+1}|\pa_{w}u(w)|^2dw}
{\int\limits_{x-1}^{x+1}|u(w)|^2dw},
\end{equation}
to conclude that, via  Fisher's min-max principle, that
\begin{equation}
 \frac{\pi^2}{4\cdot 3^{|2b-1|}}\leq \inf_{\{u:\bar{u}=0\}}  
\frac{\int\limits_{x-1}^{x+1}|\pa_{w}u(w)|^2w^{2b-1}dw}
{\int\limits_{x-1}^{x+1}|u(w)|^2w^{2b-1}dw}.
  \end{equation}
\end{enumerate}
\end{proof}

Lemma~\ref{lem3.1.006} implies the result for the $n+m$-dimensional case with
constant weights $\bb_0,$ which completes the proof of Proposition~\ref{prop3.2.002}. 
\end{proof}

We now give the proof of Theorem~\ref{prop3.2.003}. Let $E=B^i_r(w_0;y_0)$ be
the (intrinsic) ball with center 
$(w_0;y_0)$ and radius $r>0$. As noted earlier, we use Jerison's covering argument,
essentially as in~\cite[Theorem 5.3.4]{SaloffCosteLMS}. For the convenience of the reader we outline
the argument, highlighting places where our argument differs from the 
standard one. If $B= B_r(w;y),$ then for any $k > 0$, write 
\begin{equation}
  kB=B_{kr}(w;y).
\end{equation}

We let $\cF$ denote a collection of countably many $\rho^e_{\infty}$-balls in $E$  with the
following properties:
\begin{enumerate}
\item The balls $B\in \cF$ are disjoint.
\item The balls $\{2B:\: B\in\cF\}$ are a cover of $E.$
\item If $B\in\cF$, then its radius satisfies
\begin{equation}
r(B)=10^{-3}d(B,\pa E).
\end{equation}
\item There exists a constant $K$ depending only on the doubling constant so that
\begin{equation}\label{eqn132.005}
\sup_{(w;y)\in E}\#\{B\in\cF: (w;y)\in 10^2 B\}\leq K. 
\end{equation}
\end{enumerate}
Here and throughout this argument $d(\cdot,\cdot)$ should be understood as the
distance defined by the metric $\rho^e_{\infty}.$

The existence of such a `Jerison covering' $\cF$ satisfying these properties is
standard.  Several additional properties of $\cF$ are established
in~\cite{Jerison1}, and these are essential to the argument that follows.  We
are using the metric $\rho^e_{\infty}$ to define the balls in the covering. While the shortest
paths for this metric are not unique. Euclidean geodesics (i.e., straight line
segments) are length-minimizing paths for $\rho^e_{\infty}$, so, by convention,
we use these, thereby rendering the choice of shortest path unique.

Note that if $B_{\sigma}(w;y)\in\cF$, then 
\begin{equation}
  \sigma\leq 10^{-2}\dist((w;y),E^c)\leq  10^{-2}\dist((w;y),\pa S_{n,m}).
\end{equation}
Our use of Jerison's argument rests on the following lemma:
\begin{lemma}  There is a positive constant $C_1,$ so that if the ball $B = B_{\sigma}(w;y)$ has radius 
\begin{equation}
\sigma\leq  10^{-2}\dist((w;y),\pa S_{n,m}),
\end{equation}
then for any $(\tw;\ty)\in 10B,$ we have
  \begin{equation}\label{eqn133.005}
    \frac{1}{C_1}\leq \frac{\prod\limits_{j=1}^n\tw_i^{b_i(\tw;\ty)}}
{\prod\limits_{j=1}^n\tw_i^{b_i(w;y)}}\leq C_1.
\end{equation}
\end{lemma}
\begin{proof}
 Recall that there is are constant $0<\beta_0$ and $C$ so that
  for $1\leq i\leq n$ and $(w;y)\in S_{n,m}$
\begin{equation}\label{eqn134.005}
 \beta_0\leq b_i(w;y)\text{ and }
|b_i(w; y) - b_i(\tw; \ty)| \leq \frac{C}{|\log \rho_\infty^e( (w;y) , (\tw; \ty) )|}.  
\end{equation}
 Moreover there is an $R$ so that the functions $\{b_{i}(w;y):\: i=1,\dots,n\}$
are constant in $[B_R(0;0)]^c.$ 
From this is it clear that if we fix any positive number $\rho,$ then there is
a constant $C_0$ so that if $\dist((w;y),b S_{n,m})>\rho$ and $\sigma\leq
10^{-2}\dist((w;y),b S_{n,m}),$ then, for all $(\tw;\ty)\in 10B_{\sigma}(w;y)$ we
have the estimate:
\begin{equation}
   \frac{1}{C_0}\leq \frac{\prod\limits_{j=1}^n\tw_i^{b_i(\tw;\ty)}}
{\prod\limits_{j=1}^n\tw_i^{b_i(w;y)}}\leq C_0.
  \end{equation}
Thus we only need to consider balls with centers close to $\pa S_{n,m}.$ 

Let $\rho<1/10$ and assume that $\dist((w;y),b S_{n,m})<\rho.$ Let
$B_{\sigma}(w;y)$ be a ball with $\sigma\leq 10^{-2}\dist((w;y),\pa S_{n,m}).$
And let $\bar{w}=\min\{w_1,\dots,w_n\},$ clearly $\bar{w}=\dist((w;y),b S_{n,m}).$ If
$(\tw;\ty)\in 10B_{\sigma}(w;y),$ then the $w$-coordinates satisfy
\begin{equation}
  \frac{9\bar{w}}{10}<\tw_i<1,
\end{equation}
and therefore 
\begin{equation}
  |\log\tw_i|\leq|\log \frac{9\bar{w}}{10}|.
\end{equation}
The ratio in~\eqref{eqn133.005} satisfies the estimate
\begin{equation}
  \exp\left(-nC\frac{\log\left(\frac{9\bar{w}}{10}\right)}{\log\left(\frac{\bar{w}}{10}\right)}\right)
 \leq \frac{\prod\limits_{j=1}^n\tw_i^{b_i(\tw;\ty)}}{\prod\limits_{j=1}^n\tw_i^{b_i(w;y)}}\leq
 \exp\left(nC\frac{\log\left(\frac{9\bar{w}}{10}\right)}{\log\left(\frac{\bar{w}}{10}\right)}\right).
 \end{equation}
The lemma follows follows easily from these bounds.
\end{proof}
Combining this lemma with Proposition~\ref{prop3.2.002}, and the Courant-Fisher
min-max principle, we obtain the corollary:
\begin{corollary}\label{cor3.1.005} Assume that the exponents $\{b_i(w;y)\}$
satisfy $0<\beta_0\leq b_i(w;y),$ the estimate in~\eqref{regb},
and are bounded above by $B.$ Let $1\leq \kappa\leq 10.$ 
There is a constant $C$ depending on $\beta_0,\beta_1, B$ and the dimension $(n+m)$ so that 
if $B_{\sigma}(w;y)$ is a ball with 
\begin{equation}\label{eqn135.004}
\sigma\leq  10^{-2}\dist((w;y),\pa S_{n,m}),
\end{equation}
then for any $u\in\cC^1(\overline{B_{\kappa\sigma}(w;y)})$ we have the estimate
\begin{multline}
\int\limits_{B_{\kappa\sigma}(w;y)}|u(\tw;\ty)-\bar{u}|^2d\mu_{\bb}(\tw;\ty)\leq\\
 C (\kappa\sigma)^2  \int\limits_{B_{\kappa\sigma}(w;y)}|\nabla u(\tw;\ty)|^2 d\mu_{\bb}(\tw;\ty),
\end{multline}
where
\begin{equation}
\bar{u}=\frac{ \int\limits_{B_{\kappa\sigma}(w;y)}u(\tw;\ty)d\mu_{\bb}(\tw;\ty)}{\mu_{\bb}(B_{\kappa\sigma}(w;y))}.
\end{equation}
In particular, this estimate holds for any ball $B$ belonging to a covering $\cF$ as defined above.
\end{corollary}
The remaining lemmas needed to apply the argument from Section 5.3
of~\cite{SaloffCosteLMS} to prove Theorem~\ref{prop3.2.003} are given in
Appendix~\ref{appC}, where we have sketched the details of the proof using this
covering argument, in part, because it applies immediately to establish an
important generalization of this inequality wherein we include cutoff functions
in the integrals.  The proof of the following result is very similar to the one
given above and in Appendix~\ref{appC}.
\begin{proposition}\label{prop3.4}
  For $(x;y)\in S_{n,m}$ and $r>0,$ let
  \begin{equation}
    \phi(\cdot)=\max\left\{\left[1-\frac{\rho_i(\cdot,(x;y))}{r}\right],0\right\}.
  \end{equation}
Under the hypotheses of the previous theorem, there is a constant $C$ so that
\begin{equation}
\int\limits_{B^i_{r}(x;y)}|u-u_{\phi}|^2\phi^2d\mu_{\bb}(\tx,\ty)\leq C r^2\int\limits_{B^i_r(x;y)}\phi^2(\tx,\ty)
\langle A\nabla u,\nabla u\rangle d\mu_{\bb}(\tx,\ty),
\end{equation}
where
\begin{equation}
  u_{\phi}=\frac{\int\limits_{B^i_{r}(x;y)}\phi^2 ud\mu_b{\tx,\ty}}{\int\limits_{B^i_{r}(x;y)}\phi^2 d\mu_b{\tx,\ty}}.
\end{equation}
\end{proposition}
\noindent
This is Corollary 2.5 in~\cite{Sturm3}. It is needed for Moser's proof of the parabolic Harnack inequality. 

An important consequence of these results is the Sobolev inequality, Theorem 2.6 in~\cite{Sturm3}:
\begin{theorem} Let $D\geq 3$ be such that for all $0<r,$ and $(x;y)\in S_{n,m}$ we have the doubling property
\begin{equation}
\mu_{\bb}(B^i_{2r}(x;y))\leq 2^D\mu_{\bb}(B^i_{r}(x;y)).
\end{equation}
For all functions in $\cD(Q_{B^i_r(x;y)}),$ we have the estimate
\begin{multline}\label{SobIneq}
\left[\int\limits_{B^i_r(x;y)}|u|^{\frac{2D}{D-2}}d\mu_{\bb}(\tx,\ty)\right]^{\frac{D-2}{D}}\leq\\
C_S\frac{r^2}{[\mu_{\bb}(B^i_r(x;y))]^{\frac 2D}}\left[Q_{B^i_r(x;y)}(u,u)+\frac{1}{r^2}(u,u)_{\bb}\right].
\end{multline}
\end{theorem}

\begin{remark} These proofs of the doubling property and scale-invariant $L^2$
  Poincar\'e inequality readily adapt to allow the replacement of the class of
  measures $\{d\mu_{\bb}\}$ by the slightly more general class of measures of
  the form $\{e^{U(w;y)}d\mu_{\bb}(w;y)\}$, where $U(w;y)$ is a bounded
  $\cC^0$-function that is constant outside a compact set. This vindicates our
  claim, made after~\eqref{eqn90.006}, that the set of measures and quadratic
  forms to which our analysis applies is invariant under (bounded) dilations.
\end{remark}

\section{From Dirichlet Forms to Operators}\label{sect4}
The symmetric Dirichlet form $Q_B(u,v),$ with core $\cC^1(\overline{B})$
introduced above defines an unbounded self-adjoint operator, $L_Q,$ acting on a
dense domain $\cD(L_{Q})\subset L^2(B;d\mu_{\bb}).$ There are two features in
the definition of $\cD(L_{Q})$: first, the formal symbol of the operator, and
second, the natural boundary condition. The domain of the Dirichlet form
$\cD(Q)$ is the graph closure of $\cC^1(\overline{B})$ with respect to the norm
\begin{equation}
  |u|_Q^2=\|u\|^2_{L^2(B;d\mu_{\bb})}+Q_B(u,u).
\end{equation}
The domain of the operator is defined by the condition: $u\in \cD(L_Q),$ if
there is a constant $C$ so that 
\begin{equation}\label{eqn169.006}
  |Q_B(v,u)|\leq C\|v\|^2_{L^2(B;d\mu_{\bb})} \quad \mbox{for any} \quad v\in \cD(Q). 
\end{equation}
By the Riesz representation theorem there is a unique element $w\in
L^2(B;d\mu_{\bb})$ so that
\begin{equation}
  Q_B(v,u)=-( v,w)_{\bb,B}.
\end{equation}
We define $L_Qu=w.$ 

\subsection{The Second Order Operator}\label{ssec4.1.006}
By considering smooth functions in the form domain we can use the condition
in~\eqref{eqn169.006} to derive the formal symbol of the operator associated to
the symmetric Dirichlet form $Q$, along with the boundary conditions that must
be satisfied by smooth elements of $\cD(L_Q).$ These conditions are then
satisfied in a distributional sense by all elements of the operator domain.
As it fits better with our earlier work, we derive these formul{\ae} in the $(x;y)$-variables. 

After some calculation, the integration by parts gives that
\begin{multline}
  L_Qu=\Bigg[ \sum_{j=1}^n[x_j\pa_{x_j}^2u+b_j(x;y)\pa_{x_j}u]+
\sum_{i,j=1}^n\pa_{x_j}\sqrt{x_ix_j}a_{ij}\pa_{x_j}u+\\
\frac{1}{2}\sum_{i=1}^n\sum_{l=1}^m[\pa_{y_l}\sqrt{x_j}c_{jl}\pa_{x_j}u+\pa_{x_j}\sqrt{x_j}c_{jl}\pa_{y_l}u]+
\sum_{k,l=1}^m\pa_{y_m}d_{lm}\pa_{y_l}u+Vu\Bigg],
\label{lq1}
\end{multline}
where $V$ is a vector field with possibly slightly singular coefficients. For
each $j$ the formal Neumann-type boundary condition along $\pa S_{n,m}\cap
\overline{B}$ is given by
\begin{equation}
  \lim_{x_j\to 0^+}\left[x_j^{b_j}\pa_{x_j}u+x_j^{b_j-\frac 12}\left(
\sum_{i=1}^na_{ij}\sqrt{x_i}\pa_{x_i}u+\sum_{l=1}^mc_{jl}\pa_{y_l}u\right)\right]=0.
\end{equation}

To make $V$ more explicit, set 
\begin{equation}
W_{\bb}=\prod_{j=1}^nx_j^{b_j(x;y)-1}.
\end{equation}
Then 
\begin{multline}
V=\sum_{i,j=1}^n[\log x_i\pa_{x_j}b_i+\sqrt{\frac{x_i}{x_j}}a_{ji}(W_{\bb}^{-1}\pa_{x_i}W_{\bb})]x_j\pa_{x_j}
+ \sum_{k,l=1}^md_{lm}(W_{\bb}^{-1}\pa_{y_m}W_{\bb})\pa_{y_l}+\\
\frac{1}{2}\sum_{j=1}^n\sum_{l=1}^mc_{jl}[(W_{\bb}^{-1}\pa_{y_l}W_{\bb})\sqrt{x_j}\pa_{x_j}u+\sqrt{x_j}(W_{\bb}^{-1}\pa_{x_j}W_{\bb})\pa_{y_l}],
\label{lV1}
\end{multline}
where
\begin{equation}
  \begin{split}
    W_{\bb}^{-1}\pa_{x_j}W_{\bb}&=\frac{b_j-1}{x_j}+\sum_{i=1}^n\log
    x_i\pa_{x_j}b_i\\
W_{\bb}^{-1}\pa_{y_l}W_{\bb}&=\sum_{i=1}^n\log x_i\pa_{y_l}b_i.
  \end{split}
\end{equation}
A typical assumption in population genetics is that the coefficients $\{a_{ij},c_{jl}\}$ in \eqref{lq1} can be written
as $a_{ij}=\sqrt{x_ix_j}\alpha_{ij}(x;y),$ and $c_{jl}=\sqrt{x_j}\gamma_{jl}(x;y),$ where $\{\alpha_{ij},\gamma_{jl}\}$ are 
smooth functions of $(x;y)$. Thus
\begin{equation}\label{eqn176.005}
V=\sum_{j=1}^n\left[\beta^0_j+\sum_{i=1}^n\log x_i\beta^1_{ij}\right]x_j\pa_{x_j}+
\sum_{l=1}^m\left[\epsilon^0_j+\sum_{i=1}^n\log x_i\epsilon^1_{il}\right]\pa_{y_l},
\end{equation}
where $\{\beta^0_j,\beta^1_{ij},\epsilon^0_{l},\epsilon^1_{il}\}$ are smooth in $(x;y)$, so in this case $V$ is 
tangent to $\pa S_{n,m},$ but has slightly singular coefficients. Under this
hypothesis, the boundary condition along $\pa S_{n,m}\cap\overline{B}$ becomes
\begin{equation}
  \lim_{x_j\to 0^+}x_j^{b_j}\left[\pa_{x_j}u+\left(
\sum_{i=1}^n\alpha_{ij}x_i\pa_{x_i}u+\sum_{l=1}^m\gamma_{jl}\pa_{y_l}u\right)\right]=0,
\end{equation}
which is certainly satisfied if $u$ is in $\cC^1(S_{n,m}).$ Indeed, it is a
simple to show that a function $u\in\cC^2(S_{n,m})$ whose derivatives decay
rapidly enough belongs to $\Dom(L_{Q}).$

The log terms in these coefficients do not appear, at least to leading order,
if the derivatives of the weights $\{b_i(x;y)\}$ vanish along the boundary,
e.g. if these functions are constant and their gradients vanish on $\pa
S_{n,m}$. If that is the case, then 
\begin{equation}
V=\sum_{i,j=1}^n\alpha_{ij}(b_i-1)x_j\pa_{x_j}+ \sum_{l=1}^m\sum_{j=1}^n\gamma_{jl}(b_j-1)\pa_{y_l}.
\end{equation}

We now discuss two possible modifications to the form of this second order
operator which may be directly handled by our methods.  The first is to replace the
measure $d\mu_{\bb}$ by a multiple
 \begin{equation}
d\mu_{\bb,U}(x;y)=e^{U(x;y)}d\mu_{\bb}(x;y),
 \end{equation}
where, for example, $U$ is $\cC^1$ (as a function of $(x;y)$) and is constant outside of
a compact set. The extra terms coming from this factor in the integration by parts leads 
to an additional  ``conservative'' tangent vector field 
\begin{multline}
  V_U=\sum_{j=1}^n(\pa_{x_j}U)x_j\pa_{x_j}+\sum_{i,j=1}^n\sqrt{x_ix_j}a_{ij}\pa_{x_j}U\pa_{x_i}+\\
\frac{1}{2}\sum_{l=1}^m\sum_{j=1}^n\sqrt{x_j}c_{jl}\left[\pa_{y_l}U\pa_{x_j}+\pa_{x_j}U\pa_{y_l}\right]+
\sum_{k,l=1}^md_{lm}\pa_{y_{m}}U\pa_{y_l}
\label{lVU1}
\end{multline}
The associated second order operator is denoted $L_{Q,U}$. 
It is quite straightforward to incorporate such a factor into all of the arguments above and below. 

\subsection{Non-self Adjoint Perturbations}
A general Kimura operator $L$ may deviate from  
the operator $L_{Q,U}$ defined by the symmetric Dirichlet form~\eqref{eqn1.009}, with
the modified measure $d\mu_{\bb,U}$, by a first order term.
Indeed, it is typically impossible to write $L$ as in \eqref{lq1}, with $V$ a sum of
two terms \eqref{lV1} and \eqref{lVU1}.  To accommodate this, we use the formalism of 
non-symmetric Dirichlet forms. Many of the estimates proved in~\cite{SaloffCosteLMS} and~\cite{Sturm3} 
extend to operators defined in this way, and indeed~\cite{Sturm2} proves some of these.  In lieu 
of following Sturm's argument, we show that the proofs given in~\cite{SaloffCosteLMS} can be adapted 
to the present circumstance. 

To be more specific, 
consider a non self-adjoint operator 
\begin{equation}
L_{Q}^{X,c}=L_{Q,U}-V_X-c,
\end{equation}
where $V_X$ is a tangent vector field, with possibly singular coefficients and
$c$ is a measurable function. The tangent part of the vector field in $L_Q$ has
coefficients with log-singularities, and hence for $L_{Q}^{X,c}$
to be an arbitrary generalized Kimura diffusion operator, as defined
in~\cite{EpMaz2}, we must add a perturbation $V_X$ that also has 
log-singularities. Using a simple integration by
parts trick we are able to control such terms with mild singularities along the
boundary. To that end we prove the following lemma in Appendix~\ref{appC2}:
\begin{lemma}\label{lemB2.2.0101} Assume that $\bb=(b_1,\dots,b_n)$ are positive differentiable
  functions of $(\bx;\by),$ with $0<\beta_0<b_j,$ constant outside a compact set. Let $q$ be a
  measurable function defined on $S_{n,m}$ that satisfies
\begin{equation}\label{eqn220.012}
    |q(x;y)|\leq M\left[\chi_{B}(x;y)\sum_{j=1}^n|\log x_i|^k+1\right].
  \end{equation}
for some $k\in\bbN, 0<M,$ and  $B$ a bounded neighborhood of $(0;0),$ 
Given $\eta>0$ there  is a $C_{\eta}$ so that for any
$2\leq p,$ we have
\begin{equation}
  \int\limits_{S_{n,m}}|q(x;y)|u^p(x;y)d\mu_{\bb}\leq
 \eta\int\limits_{S_{n,m}}\langle A\nabla u^{\frac{p}{2}},\nabla
 u^{\frac{p}{2}}\rangle d\mu_{\bb}+
C_{\eta} \int\limits_{S_{n,m}}u^pd\mu_{\bb},
\end{equation}
for $u$ a bounded, compactly supported, non-negative function in $\Dom(Q).$
\end{lemma}

We suppose that $X$ is a
continuous $\bbR^{n+m}$-valued function in $\Int S_{n,m},$ with
$|X|^2_A=\langle AX,X\rangle$ satisfying an
estimate like in~\eqref{eqn220.012}. We define the non-symmetric Dirichlet
form
\begin{multline}\label{eqn194.006}
Q_B^X(u,v)=\int\limits_B\bigg\{\sum_{j=1}^n[x_i\pa_{x_i}uX_i+\sum_{i,j=1}^n+
x_ix_ja_{ij}(x;y)\pa_{x_i}u X_j+\\
\frac{1}{2}\sum_{i=1}^n\sum_{l=1}^mx_ic_{il}[\pa_{x_i}uX_l+\pa_{y_l}uX_i]+
\sum_{k,l=1}^md_{kl}(x;y)\pa_{y_k}u X_l\bigg\}v\times\\
x_1^{b_1(x;y)-1}\cdots x_n^{b_n(x;y)-1}dxdy;
\end{multline}
this represents the action of the vector field $V_X,$  which is continuous and tangent to the
boundary.

Representing the integrand in~\eqref{eqn194.006} as $\langle A\nabla u,Xv\rangle
d\mu_{\bb},$ and allowing also for a zeroth order term $cu,$ where $c$ is a
measurable, real valued function, satisfying an
estimate like that in~\eqref{eqn220.012},
we define 
\begin{equation}
\tQ^{X,c}_{U,B}(u,v)=\int\limits_{B}\left[\langle A\nabla u,\nabla
  v\rangle+\langle A\nabla u,Xv\rangle + c uv\right]
e^Ud\mu_{\bb}.
\end{equation}
For simplicity of notation, and because it provides no additional generality, we
shall omit the factor $e^U$ in the measure.

A ``sector condition'' holds for $\tQ^{X,c}_B$: there is a constant $C> 0$ so
that for any $u,v\in\cD(Q)$,
\begin{equation}\label{eqn118.010}
|\tQ^{X,c}(u,v))|\leq C\left(Q(u,u)+(u,u)_{\bb}\right)^{\frac 12}\left(Q(v,v)+(v,v)_{\bb}\right)^{\frac 12}.
\end{equation}
This is clear since the Cauchy-Schwarz inequality implies that
\begin{equation}
\begin{split}
|\tQ^{X,c}(u,v))|&=\left|\int\left[\langle A\nabla u,\nabla  v+vX\rangle + cuv\right]d\mu_{\bb}\right|\\
&\leq \sqrt{2}\left(Q(u,u)+(u,u)_{\bb}\right)^{\frac 12}(Q(v,v)+(\sigma v,v)_{\bb})^{\frac 12},
  \end{split}
\end{equation}
where $\sigma =|X|_A^2 + |c|^2.$ By Lemma~\ref{lemB2.2.0101},there is a $C'$
so that
\begin{equation}
 (\sigma v,v)_{\bb}\leq C'\left[Q(v,v)+(v,v)_{\bb}\right],
\end{equation}
which proves~\eqref{eqn118.010}.  From this it is immediate
that the form domains of  $\tQ^{X,c}$ and $Q$ agree. A function
$u\in\Dom(\tQ^{X,c})$ is in the domain of the  operator $L^{X,c}_Q$ if there is
a constant $C$ so that, for every $v\in \Dom(\tQ^{X,c})$, 
\begin{equation}
  |\tQ^{X,c}(u,v)|\leq C\|v\|_{\bb}.
\end{equation}
This implies, as before, that there is a unique element, $w\in L^2$ so that
\begin{equation}
  \tQ^{X,c}(u,v)=-(w,v)_{\bb};
\end{equation}
we then define $L^{X,c}_Qu=w.$

For the associated operator to satisfy the Markov property, and hence define contractions on 
$L^p$-spaces, (see Lemmas 1.4 and 1.5 in~\cite{Sturm2}) we would need to assume that
\begin{equation}\label{eqn116.0010}
  c-\frac{1}{2}\Div_{A,\bb}X\geq 0,
\end{equation}
where
\begin{multline}
  \Div_{A,\bb}X=\nabla_{x;y}\cdot(AX)+\\
\bx^{1-\bb}\left[\sum_{i,k}A_{ik}X_k\pa_{x_i}(\bx^{\bb-1})+
\sum_{l,k}A_{(l+n)k}X_k\pa_{y_l}(\bx^{\bb-1})\right].
\end{multline}
Writing out the second line in detail gives
\begin{equation}\label{eqn120.0010}
  \sum_{i,k}A_{ik}X_k\left[\frac{b_i-1}{x_i}+\sum_{s=1}^{n}\log
    x_s\pa_{x_i}b_s\right]+
\sum_{l,k}A_{(l+n)k}X_k\left[\sum_{s=1}^n\log x_s\pa_{y_l}b_s\right],
\end{equation}
where $1\leq i\leq n,$ $1\leq l\leq m$ and $1\leq k\leq m+n.$ For~\eqref{eqn116.0010} to 
hold with a bounded function $c,$ it is generally necessary that the following three conditions hold:
\begin{enumerate}
\item The weights must be constant along the appropriate boundary components.
\item The vector field $AX$ must be Lipschitz.
\item The coefficients $A_{ik}$ are (boundedly) divisible by $x_i.$
\end{enumerate}
For our applications, these hypotheses are unnatural, but fortunately they are actually not
necessary. Using the estimates that follow from Lemma~\ref{lemB2.2.0101}, with $p=2,$ we easily establish that there is a
constant $m$ so that a weak local solution, in $[0,T]\times B,$
to $u_t-L_Q^{X,c}u=0$ satisfies,
\begin{equation}\label{eqn135.014}
  \|u(t)\|_{L^2(B;d\mu_{\bb})}\leq e^{m t}\|u(0)\|_{L^2(B;d\mu_{\bb})}.
\end{equation}

The operator adjoint to the one defined by $\tQ^{X,c},$ with respect to  the $L^2(B;d\mu_{\bb})$-pairing, is 
\begin{equation}
\hL^{X,c}_Q=L_Q+V_X-\hc;
\end{equation}
this is defined by the Dirichlet form $\tQ^{-X,\hc}$, where
\begin{equation}\label{eqn122.0010}
  \hc=c-\Div_{A,\bb}X.
\end{equation}
With this representation for the adjoint, $ \hL^{X,c}_Q$ has an unbounded term of order zero, 
even if $c=0,$ unless the weights are constant. In 
the sequel we prove Harnack estimates for the operators $L_Q^{X,c}$ assuming
that $c$ is bounded and that $X$ satisfies an estimate like that
in~\eqref{eqn220.012}.  This enables us to prove the Harnack estimate and the
H\"older continuity for solutions to a generalized Kimura diffusion on a
compact manifold with corners, with initial data in $L^2.$ Since
Lemma~\ref{lemB1.010} holds for potentials with $\log$-singularities at the
boundary, we can  use the argument in~\cite{SaloffCosteLMS} to prove upper bounds for
the heat kernel in this more general case.

\subsection{Consequences of the Doubling Property and Poincar\'e Inequality}\label{sec4.2.001}
The hypothesis that the functions $b_i(x;y)$ and $U(x;y)$ are constant outside a compact set implies
that the doubling property and Poincar\'e inequality hold globally in $S_{n,m}.$ Our main intention, however, 
is to apply these results to solutions of the parabolic equation
\begin{equation}
  \pa_tu-Lu=0\text{ with } \left. u \right|_{t=0} = f
\end{equation}
on a compact manifold with corners $P$, where $L$ is a generalized Kimura diffusion operator. We thus
work in a boundary adapted coordinate system, and use the fact that these estimates 
hold for local solutions. Grigor'yan and Saloff-Coste, and in somewhat greater generality 
Sturm, show that the doubling property,~\eqref{eqn78.0003},  of the measure, and the scale-invariant 
$L^2$ Poincar\'e inequality,~\eqref{eqn179.005}, imply a range of properties of solutions to
both the parabolic and elliptic problems, including:  
\begin{enumerate}
\item Harnack inequalities for non-negative solutions.
\item H\"older continuity for weak solutions with initial data in
  $L^2(B;d\mu_{\bb}).$ 
\item Pointwise upper and lower bounds for the heat kernel itself.
\end{enumerate}

It is shown in~\cite{Sturm2} that if a Dirichlet form satisfies the hypotheses 
of uniform parabolicity (UP) and strong uniform parabolicity (SUP), as well as
the doubling property for $d\mu_{\bb}$ and the scale invariant $L^2$-Poincar\'e
inequality, established for $Q$ in Theorem~\ref{prop3.2.003}, then weak solutions
satisfy the conclusions of Lemma 1 in~\cite{Moser2}.  Sturm did not derive all
the conclusions that are available in the non-symmetric case, and the
verification of the SUP condition requires the assumption that the weights
$\{b_i(x;y)\}$ are constant and their gradients vanish at $\pa P$.  We therefore 
show directly that analogues of Theorems 5.2.9, 5.2.16 and Lemma 5.4.1 in~\cite{SaloffCosteLMS} 
hold for weak local solutions, i.e., $u\in \cD(Q_{B^i_r(x;y)})$ for which
\begin{equation}
  \tQ^{X,c}_{B^i_r(x;y)}(u,\varphi)=0
\end{equation}
for all $\varphi\in\cD(Q_{B^i_r(x;y)})\cap L^{\infty},$ with support in
$B^i_r(x;y).$ The proof of the Harnack inequality then follows, more or less
functorially, from the argument in~\cite{SaloffCosteLMS}, which employs the
lemma of Bombieri and Giusti (Lemma 2.2.6 in~\cite{SaloffCosteLMS}).  The
proofs of the necessary lemmas are given in Appendix~\ref{appC2}.  Here we simply
state the consequences of these estimates. The first and most important is a Harnack inequality for
local solutions.
\begin{theorem}Suppose that the functions $\{b_i(x;y):\:i=1,\dots,n\}$ defined
  in $\bbR_+^n\times\bbR^m$ are continuously differentiable functions of
  $(\bx;\by),$ bounded below by a positive constant and constant outside
  a compact set, and $X(w;y)$
  is a continuous $\bbR^{n+m}$-valued function, satisfying~\eqref{eqn220.012}
  for some $k\in\bbN,$ and $B$ a bounded set, which vanishes outside of a
  compact set, and $c(w;y)$ is a bounded measurable function supported in a
  compact set.  There is a constant $C$ so that for any $0<r,$ and
  $(x;y)\in\bbR_+^n\times\bbR^m,$ and $u$ a non-negative, weak solution to
\begin{equation}
  \pa_tu=(L_Q-V_X-c)u
\end{equation}
 in $W=(t-4r^2,t) \times
  B^i_{2r}(x;y),$ we have the following estimate:
\begin{equation}
  \sup\limits_{W^-} u\leq C\inf\limits_{W_+}u
\end{equation}
where 
\begin{equation}
  \begin{split}
    W^+&=\left(t-r^2,t\right)\times B^i_{r}(x;y)\\
 W^-&=\left(t-3r^2,t-2r^2\right)\times B^i_{r}(x;y).
  \end{split}
\end{equation}
\end{theorem}
\begin{proof}[Sketch of proof] Using Lemmas~\ref{lemB1.010},~\ref{lemc2.01}
  and~\ref{lemc3.01} we verify that the hypotheses of the lemma of Bombieri and
  Giusti (Lemma 2.2.6 in~\cite{SaloffCosteLMS}) are satisfied with
  $\alpha_0=\infty.$ The proof of the inequality then follows exactly as
  in~\cite{Moser2}, which is essentially identical to the argument used
  in~\cite{SaloffCosteLMS}. 
\end{proof}
As noted above this estimate has a wide range of consequences, among them the
H\"older continuity of solutions to the initial value problem for the parabolic
operator
$\pa_t-(L_Q-V_X)$ and upper and lower bounds on the heat kernel. We first state the
H\"older continuity result.
\begin{corollary}\label{cor5.1.003} If $\{b_i(x;y)\}$ are positive $\cC^1$-functions of
  $\{\bx,\by\},$ which are constant outside of compact set, and $X(w;y)$
  is a continuous $\bbR^{n+m}$-valued function, satisfying~\eqref{eqn220.012} for
  some $k\in\bbN,$ which vanishes outside of a compact set, then there exists a
  $\gamma>0$ and a constant $C$ such that, for all balls $B^i_{2r}(x;y)\subset
  S_{n,m}$ and all $t\in\bbR,$ if $u$ is a weak solution to the equation
  \begin{equation}
    \pa_tu-(L_Q-V_X)u=0
  \end{equation}
in the the set $W=(t-4r^2,t)\times B^i_{2r}(x;y),$ then for $(s_1,x_1;y_1),
(s_2,x_2;y_2)\in (t-r^2,t)\times B^i_r(x;y)$, 
\begin{equation}\label{eqn165.008}
  |u(s_1,x_1;y_1)-u(s_2,x_2;y_2)|\leq C\sup_{W}|u|
\left(\frac{|s_1-s_2|^{\frac 12}+\rho_i((x_1;y_1),(x_2;y_2))}{r}\right)^{\gamma}.
\end{equation}
\end{corollary}
\begin{remark} Note that we need to take $c=0,$ as the proof requires that
  constant functions be solutions of the parabolic equation.
\end{remark}

\noindent
The proof exactly follows the proof of Theorem 5.4.7 in~\cite{SaloffCosteLMS}.

This corollary has a very useful corollary itself, which gives the rate at which
the $\cC^{0,\gamma}_{\WF}$-norm of a solution with initial data in $\cC^0$ blows
up. 
\begin{corollary}\label{cor4.3.007}
  Suppose that $u$ is defined in $W=(0,t)\times B^i_{2}(0;0),$ with $t<\frac
  12,$ satisfies
  the estimate in~\eqref{eqn165.008}. There are constants
   $C, C'$ independent of $u$ and $t$ so that if $\sup_{(x;y,t)\in
    W}|u(x;y,t)|<M,$ then
  \begin{multline}\label{eqn144.008}
    |u(t,x_1;y_1)-u(t,x_2;y_2)|\leq\\
    M\left[C\left(\frac{\rho_i((x_1;y_1),(x_2;y_2))}{\sqrt{t}}\right)^{\gamma}+
C'\left(\frac{\rho_i((x_1;y_1),(x_2;y_2))}{\sqrt{t}}\right)\right],
  \end{multline}
for $(x_1;y_1),(x_2;y_2)\in B^i_1(0;0).$ 
\end{corollary}
\begin{remark} 
As noted after  hypothesis~\eqref{eqn90.006}, the conditions under which this corollary holds 
are dilation invariant. This result also shows that 
\begin{equation}
\int\limits_{0}^{\frac 12}\|u(\cdot,t)\|_{\WF,0,\gamma}dt<C \|u\|_{L^{\infty}},
\end{equation}
which in turn implies that if $L_Q-V_X$ is a Kimura diffusion operator on the compact manifold
with corners $P,$ then its graph closure on $\cC^0(P)$ has a compact resolvent. 
\end{remark}
\begin{proof} From~\eqref{eqn165.008} it follows that for points
$(x_1;y_1),(x_2;y_2)\in B^i_1(0;0)$ with
$\rho_i((x_1;y_1),(x_2;y_2))<\sqrt{t/3},$ we have
\begin{equation}
|u(t,x_1;y_1)-u(t,x_2;y_2)|\leq
MC\left(\frac{\rho_i((x_1;y_1),(x_2;y_2))}{\sqrt{t}}\right)^{\gamma}.
\label{estdif}
\end{equation}
If $\rho_i((x_1;y_1),(x_2;y_2))>\sqrt{t/3},$ then we can choose points $(x^{(j)};y^{(j)})$, $j=0,\dots,N,$ 
on a length minimizing geodesic joining   $(x^{(0)};y^{(0)})=(x_1;y_1),$ to the
point $(x^{(N)};y^{(N)})=(x_2;y_2),$ with
\begin{equation}
\rho_i((x^{(j)};y^{(j)}),(x^{(j+1)};y^{(j+1)}))=\sqrt{\frac{t}{3}}\text{ for }j=0,\dots,N-1,
\end{equation}
and 
\begin{equation}
  \rho_i((x^{(N-1)};y^{(N-1)}),(x^{(N)};y^{(N)}))\leq\sqrt{\frac{t}{3}}.
\end{equation}
Clearly
\begin{equation}
  N\leq \frac{4\rho_i((x_1;y_1),(x_2;y_2))}{\sqrt{t}}.
\label{estN}
\end{equation}
Applying \eqref{estdif} for a sequence of points along the straight line geodesic from $(x_1;y_1)$ to $(x_2;y_2)$ gives
\begin{equation}
\begin{split}
   |u(t,x_1;y_1)-u(t,x_2;y_2)|&\leq
    MC\sum_{j=0}^{N-1}\left(\frac{\rho_i((x^{(j)};y^{(j)}),(x^{(j+1)};y^{(j+1)}))}{\sqrt{t}}\right)^{\gamma}\\
&\leq
    MCN^{1-\gamma}
\left(\sum_{j=0}^{N-1}\frac{\rho_i((x^{(j)};y^{(j)}),(x^{(j+1)};y^{(j+1)}))}{\sqrt{t}}\right)^{\gamma}.
\end{split}
\end{equation}
Then \eqref{estN} and the fact that the points $\{(x^{(j)};y^{(j)})\}$ lie along a length minimizing geodesic, 
so the sum of the $\rho_i$ distances between them telescope, shows that 
\begin{equation}
   |u(t,x_1;y_1)-u(t,x_2;y_2)|\leq MC'\left(\frac{\rho_i((x_1;y_1),(x_2;y_2))}{\sqrt{t}}\right);
\end{equation}
this completes the proof of the corollary.
\end{proof}

The final corollaries are upper and lower bounds for the heat kernel itself.
The upper bound holds for the general class of operators $L_Q-V_X-c$ we have
been considering, provided that the adjoint operator,
$\hL^{-X,\hc}_Q=L_Q+V_X-\hc,$ where $\hc$ is given by~\eqref{eqn122.0010}, is
an operator of the same type. The lower bounds only apply to the self adjoint
case.  The solution operator for the heat equation defines a semigroup
$f\mapsto T_tf,$ which is represented by a kernel function
\begin{equation}
  T_tf(x;y)=\int\limits_{S_{n,m}}p_t((x;y),(\tx;\ty))f(\tx;\ty)d\mu_{\bb}(\tx;\ty).
\end{equation}
As shown in Lemma 1.5 and Proposition 2.3 of~\cite{Sturm2}, the operator
\begin{equation}
  S_tf(\tx;\ty)=\int\limits_{S_{n,m}}p_t((x;y),(\tx;\ty))f(x;y)d\mu_{\bb}(x;y)
\end{equation}
gives the semigroup for adjoint operator $\hL^{-X,\hc}_Q.$
\begin{remark} Notice that the inclusion of the weight defining the measure 
  gives a kernel of the form
  \begin{equation}
    p_t((x;y),(\tx;\ty))\tx_{1}^{b_1(\tx;\ty)-1}\cdots \tx_{n}^{b_n(\tx;\ty)-1},
  \end{equation}
which exactly mirrors the kernels that arise in the model case. The upper bound
in~\eqref{eqn207.005} shows that the principal singularity of the heat kernel at
the incoming face is no worse than that defined by the weight function,
$\bx^{\bb-1}.$ For the self adjoint case, the lower bound~\eqref{eqn148.012}
shows that this precisely captures the leading singularity.
\end{remark}
The following estimates contain the doubling constant $D$ from~\eqref{eqn78.0003}.
The proofs of Theorem 5.2.10 and Corollary 5.2.11 in~\cite{SaloffCosteLMS} give the upper bound:
\begin{corollary}\label{cor5.2} Assume that the $\{b_i\}$ are positive $\cC^1$ functions of
  $\{\bx, \by\},$ which are constant outside a bounded set, $X(x;y)$ is a
  $\cC^1$-function, satisfying~\eqref{eqn220.012} for some $k\in\bbN,$
  vanishing outside a compact set; and $c$ is a function also
  satisfying~\eqref{eqn220.012}, for some $k'\in\bbN,$ and vanishing outside a
  compact set. For any $0<\eta$ there are constants $C_0,C_1$ so that, for all
  $t>0$ and pairs $(x;y),(\tx;\ty)\in S_{n,m},$ we have
\begin{multline}\label{eqn207.005}
 p_t((x;y),(\tx;\ty))\leq \\ \frac{C_0\exp\left(-\frac{\rho_i^2((x;y),(\tx;\ty))}{4(1+\eta)t}\right)}
{\sqrt{\mu_{\bb}(B^i_{\sqrt{t}}(x;y))\mu_{\bb}(B^i_{\sqrt{t}}(\tx;\ty))}}\times
\left(1+\frac{\rho_i((x;y),(\tx;\ty))}{\sqrt{t}}\right)^{D}
\cdot\exp(C_1t).
\end{multline}
If $X=c=0,$ then we can take $C_1=\eta=0$ in this estimate.
\end{corollary}
\begin{proof} The proof given in~\cite{SaloffCosteLMS} for Corollary 5.2.11
  applies with several modifications. If $X=0,$ then
the kernel function is symmetric and, for $t>0$ defines a weak solution to
$(\pa_t-L_Q^c)u=0,$ in both the $(t,x;y)$ and $(t,\tx;\ty)$ variables. We can
therefore apply the estimates in~\eqref{eqn245.009} with $p=1$ in both sets of
variables. 

If $X$ is not zero, then the kernel weakly satisfies the equations
\begin{equation}
  \begin{split}
    &(\pa_t-L^{X,c}_{Q,(x;y)})p_t=0\\
&(\pa_t-\hL^{X,c}_{Q,(\tx;\ty)})p_t=0.
  \end{split}
\end{equation}
The vector field $X$ and potential $c$ are allowed to have log-singularities along the
boundary. The function $\hc$ defined in~\eqref{eqn122.0010}  also
satisfies the estimate in~\eqref{eqn220.012}, with some $k''\in\bbN.$ The adjoint
operator is therefore defined by a Dirichlet form satisfying the
hypotheses of Lemma~\ref{lemB1.010}, and therefore weak solutions of
$(\pa_t-\hL^{X,c}_Q)u=0$ also satisfy the estimates in~\eqref{eqn245.009} with $p=1.$

Instead of the estimate~\cite[Lemma 4.2.1]{SaloffCosteLMS}
  for the $L^2$ operator norm for $T_t^{\alpha,\phi}=e^{-\alpha\phi}T_te^{\alpha\phi}$,
  we have, for any $\epsilon>0,$ that
  \begin{equation}\label{eqn147.012}
    \|T_t^{\alpha,\phi}\|_{2\to 2}\leq e^{((1+\epsilon)\alpha^2 +C_{\epsilon})t},
  \end{equation}
for a constant $C_{\epsilon}.$  To see this we observe that if
$u(t)=T_t^{\alpha,\phi}f,$ then
\begin{multline}
  \pa_t\|u(t)\|_{\bb}^2=\\
-\int\limits_{S_{n,m}}\left[\langle A\nabla u,\nabla u\rangle+\langle A\nabla
  u,X u\rangle+
(\alpha\langle A\nabla \phi,X\rangle-\alpha^2\langle A\nabla \phi,\nabla \phi\rangle+c)u^2\right]d\mu_{\bb}.
\end{multline}
Using the Cauchy-Schwarz and arithmetic-geometric mean inequalities we see that
\begin{multline}
  \pa_t\|u(t)\|_{\bb}^2\leq 
-\frac{1}{2}\int\limits_{S_{n,m}}\langle A\nabla u,\nabla u\rangle d\mu_{\bb}+
\alpha^2\int\limits_{S_{n,m}}\langle A\nabla \phi,\nabla \phi\rangle
u^2d\mu_{\bb}+\\
\int\limits_{S_{n,m}}\left[|\alpha||X|_Au^2+\frac{1}{2}\langle AX,X\rangle+|c|)u^2\right]d\mu_{\bb}.
\end{multline}
Applying Lemma~\ref{lemB2.2.0101} we easily show that for any $\epsilon>0$
there is a constant $C_{\epsilon}$ so that
\begin{equation}
   \pa_t\|u(t)\|_{\bb}^2\leq [(1+\epsilon)\alpha^2+C_{\epsilon}]\|u(t)\|_{\bb},
\end{equation}
thus verifying~\eqref{eqn147.012}. If $X=c=0,$ then clearly we can take
$\epsilon=0,$ and $C_{\epsilon}=0.$
Apart from these modifications, the proof works exactly as in Saloff-Coste.
\end{proof}

The lower bound, which follows from Corollary 4.10 in~\cite{Sturm3} is somewhat
less general.
\begin{corollary}\label{cor5.3}
  Suppose that $\{b_i\}$ are positive $\cC^1$ functions of $\{\sqrt{\bx},
  \by\},$ which are constant outside a bounded set, and $L_Q$ is a generalized
  Kimura diffusion operator defined by a symmetric Dirichlet form satisfying the
  hypotheses above. If $p_t$ denotes the heat kernel for $\pa_t-L_Q,$ then
  there is a constant $C$ so that, for all $t>0$ and $(x;y),(\tx;\ty)\in
  S_{n,m}$ we have
  \begin{equation}\label{eqn148.012}
    p_t((x;y),(\tx;\ty))\geq
    \frac{\exp\left(-C\frac{\rho_i^2((x;y),(\tx;\ty))}{t}\right)}
{C\mu_{\bb}(B^i_{\sqrt{t}}(x;y))}.
  \end{equation}
\end{corollary}
\begin{remark} The proof of the off-diagonal lower bound follows from the
  Harnack inequality and a lower bound for $p_t((x;y),(x;y)).$ This diagonal
  estimate relies on the semi-group property \emph{and} the self adjointness of the
  heat kernel with respect to the measure $d\mu_{\bb}.$ Generalizations of
  these lower bounds to non-self adjoint operators are given
  in~\cite{StrookProbPDE}; we will return to this question in a later publication.
\end{remark}

\section{Applications to Population Genetics}
The foregoing results have many applications to models in population genetics.
Let $P$ be a manifold with corners, and $L$ a generalized Kimura diffusion
operator defined on $P.$ As shown in~\cite{EpMaz2}, we can introduce adapted
coordinates near each boundary point so that the operator takes the
form~\eqref{eqn1.001}.  Let $\{H_i:i=1,\dots, I\}$ denote the hypersurface
boundary components of $P.$ As shown in Proposition~\ref{prop1.001}, the
coefficient $b_i(x;y)$ of the vector field transverse vector to $H_i$, in
adapted coordinates, has a natural meaning along each $H_i.$

The principal symbol of $L$, and hence the second order part of this operator,
are globally defined throughout $P$.  The weight functions are invariantly
defined by $L$ along the faces of $P$. In~\cite{EpMaz2}, we prove a tubular
neighborhood theorem for each face of $P$, which implies that the weight
functions have global extensions to non-negative functions on $P,$ which can be
taken to be positive if the weights themselves are. Throughout this section we
assume that the weights are strictly positive. If a weight is constant, then it
can be extended to be globally constant.  As explained at the end of
Section~\ref{sec2}, these extended weights define a measure
$d\mu_{L}$, which is locally of the form~\eqref{eqn3.004} and, in each
coordinate chart, satisfies
\begin{equation}
  C^{-1}\leq\frac{d\mu_{L}}{d\mu_{\bb}}\leq C.
\end{equation}

The principal symbol of $L,$ $q^{(2)}_L,$ is a non-negative quadratic form on
the fibers of $T^*P.$ Its canonical dual defines an incomplete Riemannian
metric on $P,$ as discussed in Chapter 2 of~\cite{EpMaz2}. We denote the
distance between points $\xi,\eta\in P$ defined by this metric by
$\rho_i(\xi,\eta),$ which is consistent with our  usage of this
notation in Sections~\ref{sec3}--\ref{sect4}.  The compactness of $P$ and Proposition~\ref{prop3.1.0004}
together imply that there is a constant $D$ so that, for $\xi\in P$ and $0<r,$ 
\begin{equation}\label{eqn164.014}
  \mu_L(B^i_{2r}(\xi))\leq 2^D  \mu_L(B^i_{r}(\xi)),
\end{equation}
i.e., $\mu_L$ is a doubling measure.

Using $d\mu_L$ and the principal symbol of $L$, we can define a Dirichlet form $Q$ with core 
$\cC^1(P):$ 
\begin{equation}
Q(u,v)=\int\limits_{P}q^{(2)}_L(du,dv)d\mu_{L}.
\end{equation}
For an open set $B\subset P,$ we use the notation $Q_B$ for 
\begin{equation}
Q_B(u,v)=\int\limits_{B}q^{(2)}_L(du,dv)d\mu_{L},\text{ for }u,v\in\cC^1(\overline{B}).
\end{equation}
If $B$ is contained in an adapted local coordinate chart, then $Q_B(u,v)$ takes
the form given in~\eqref{eqn1.009}, with the measure $d\mu_{\bb}$ replaced by
$e^Ud\mu_{\bb},$ for a smooth function $U.$ Using Jerison's covering argument
and the scale invariant $L^2$-Poincar\'e inequality, with a uniform constant, for 
sufficiently small balls, we can show that there is a constant $C_P$ so that for any $\xi\in P,$ $0<r,$ and
$u\in\cC^1(B_r^i(\xi))$ we have the estimate
\begin{equation}\label{eqn167.014}
  \int\limits_{B_r^i(\xi)}|u-\bar{u}|^2d\mu_L\leq C_Pr^2Q_{B_r^i(\xi)}(u,u),
\end{equation}
where
\begin{equation}
  \bar{u}=\frac{\int\limits_{B_r^i(\xi)}ud\mu_L}{\mu_L(B_r^i(\xi))}.
\end{equation}

This Dirichlet form defines a second order, self-adjoint operator $L_{Q}$ by 
\begin{equation}
  Q(u,v)=-(L_{Q}u,v)\text{ for all }u\in\Dom(L_Q)\text{ and }v\in\Dom(Q).
\end{equation}
The difference $L_Q-L$ is a vector field $V,$ tangent to the boundary of $P.$ If the weights
are constant, then $V$ is a smooth tangent vector field, but otherwise it has
logarithmically divergent coefficients as in~\eqref{eqn176.005}. In any case
there is a globally defined section $\Xi$ of $T^*P$ so that $L$ is the second order
operator defined by the non-symmetric Dirichlet form
\begin{equation}
  \tQ^{\Xi}(u,v)=Q(u,v)+\int\limits_{P}q^{(2)}_L(du,v\Xi)d\mu_{L}.
\end{equation}
The ellipticity hypotheses imply that $q_L^{(2)}(\Xi,\Xi)^{\frac 12}$ diverges at
worst logarithmically at $\pa P.$ We let $L^{\Xi}_Q$ denote the unbounded
operator on $L^2(P;d\mu_L)$ defined by this Dirichlet form. The addition of
such a vector field does not change the natural boundary condition that appears
in the definition of domain of $L^{\Xi}_Q,$ and it again follows that functions
in $\cC^2(P)$ automatically belongs to $\Dom(L^{\Xi}_Q).$

\subsection{Regularity Results}
We begin our analysis of regular solutions to the Cauchy problem for
$\pa_tu-Lu=0$ by considering the local regularity for solutions  with initial
data in $L^2(P),$ and then $\cC^0(P).$  The following can be deduced from
the results of Section~\ref{sect4}.
\begin{theorem}\label{thm5.1.001}
  Let $P$ be a compact manifold with corners and $L$ a generalized Kimura diffusion
  operator with smooth coefficients defined on $P.$ Suppose that the weights
  defined by $L$ are positive  along every boundary
  component. If $u$ is a weak solution to the initial value problem
  \begin{equation}
    (\pa_t-L)u=0\text{ with }u(\xi,0)=f(\xi)\in L^2(P,d\mu_{L}),
  \end{equation}
then $u\in\CI(P\times (0,\infty)).$
\end{theorem}
\begin{proof} Let $\{(\phi_j,U_j)\}$ be a cover of $P$ by adapted coordinate
  charts, where $\phi_j:U_j\to W_j\subset S_{n_j,m_j}.$ In each coordinate
  chart there is a measure and Dirichlet form $Q_j$ defined in a neighborhood
  of $(0;0)$ so that, in this chart, the operator $L$ is of the form
  $L_{Q_j}-V_{X_j}.$ As $L$ is assumed to have smooth coefficients, the
  coefficient functions $X_j$ satisfy an estimate of the form given
  in~\eqref{eqn220.012}, with $k=1.$ The measure is defined throughout
  $S_{n_j,m_j},$ and the Dirichlet form can be extended as well.  The operator
  $L_{Q_j}$ is a model operator outside of a compact neighborhood of $W_j.$
  Corollary~\ref{cor5.1.003} implies that for $t > 0,$ the solution
  $\phi_j^*u(\cdot,t)$ belongs to the H\"older space
  $\cC^{0,\gamma}_{\WF}(W_j)$ for each $j.$ This shows that $u(\cdot,
  t)\in\cC^{0,\gamma}_{\WF}(P)$ for $t>0.$ We can therefore apply Corollary
  11.2.2 from~\cite{EpMaz2} to conclude that $u\in\CI(P\times (0,\infty)).$
\end{proof}

We now use the estimates in Corollary~\ref{cor4.3.007} along with the
maximum principle to conclude that there exist constants $0<\gamma<1,$ and $C,$
so that if $u(\xi,t)$ is the regular solution to $(\pa_t-L)u=0,$ in $P\times
(0,\infty)$ with $u(\xi,0)=f(\xi)\in\cC^0(P),$ then
\begin{equation}
  \|u(\cdot,t)\|_{\WF,0,\gamma}\leq C\|f\|_{\cC^0} (t^{-\gamma/2} + t^{-1/2}). 
\label{smallt}
\end{equation}
Using this for small $t$ and the estimates in~\cite{EpMaz2} for $t \gg 0$, we see that 
when $\mathrm{Re}\, \mu > 0$, 
\begin{equation}
  (\mu-L)^{-1}f=\int\limits_{0}^{\infty}e^{-\mu t}u(\cdot,t)dt\in\cC^{0,\gamma}_{\WF}(P).
\end{equation}
In fact there is a constant $C_{\mu}$ so that 
\begin{equation}
  \|(\mu-L)^{-1}f\|_{\WF,0,\gamma}\leq C_{\mu}\|f\|_{\cC^0},
\end{equation}
which leads immediately to the following. 
\begin{corollary}\label{cor5.1.008}
  Let $P$ be a compact manifold with corners and $L$ is a generalized Kimura
  diffusion operator defined on $P$ with positive weights. If $\overline{L}$ is
  the $\cC^0$-graph closure of $L$ acting on $\cC^3(P),$ then for $\mu$ with
  $\Re\mu>0,$ the resolvent operator $(\mu-\overline{L})^{-1}$ is  bounded
  from $\cC^0(P)$ to $\cC^{0,\gamma}_{\WF}(P),$ and is therefore a compact
  operator. For initial data in $\cC^0(P),$ the regular solution to the initial
  value problem $\pa_t u-Lu=0$ has an analytic extension to $\{t:\Re t>0\}.$
  The spectrum, $\sigma_{\cC^0}(\overline{L}),$ lies in a conic neighborhood of
  $(-\infty,0].$
\end{corollary}
\begin{proof} Since $u(\cdot,\epsilon)\in\cC^{0,\gamma}_{\WF}(P),$ for any
  $\epsilon>0,$ Theorem 11.2.1 in~\cite{EpMaz2} shows that $u$ extends to be
  analytic in sets of the form $\{t:\Re t>\epsilon\}.$ The analyticity
  assertion follows from this. Since $\overline{L}$ has a compact resolvent,
  every point in $\sigma_{\cC^0}(\overline{L})$ is an eigenvalue. The eigenvectors
  belong to $\CI(P),$ and therefore the spectrum of $\overline{L}$ is the same
  as its spectrum acting on $\cC^{0,\gamma}_{\WF}(P).$ In Theorem 11.1.1
  of~\cite{EpMaz2} this is shown to lie in a conic neighborhood of
  $(-\infty,0].$
 \end{proof}

\subsection{Heat Kernel Estimates}
As noted above, the kernel function for the semigroup 
 $e^{tL^{\Xi}_Q}$ takes the form 
 $$p_t(\xi,\eta)d\mu_{L}(\eta).$$
 Corollary~\ref{cor5.2} indicates that it is reasonable to expect that the
 kernel $p_t$ is a bounded function for $t>0.$ The kernel function locally
 satisfies the equations
\begin{equation}\label{eqn158.012}
  (\pa_t-L_{Q,\xi}+V_{X,\xi})p_t(\xi,\eta)=0\text{ and }
(\pa_t-L_{Q,\eta}+\hV_{X,\eta})p_t(\xi,\eta)=0.
\end{equation}
The adjoint of the vector field, $\hV_{X,\eta},$ is computed with respect to the measure
$d\mu_{L}.$ It is of the form $-V_X+\hc,$ where, in local coordinates, 
$ \hc= \Div_{A,\bb}X.$
The operator $\hL=L_{Q,\eta}-\hV_{X,\eta}$ is the $L^2$-adjoint of
$L=L_{Q,\xi}-V_{X,\xi}.$

It is important to note that this representation for the adjoint operator is
different from the one employed in~\cite{EpMaz2}. In this paper the semigroup
acts on $L^2(P;d\mu_{L}),$ and the adjoint $\hL$ is defined with respect to
this Hilbert space structure. The operator $\overline{L}$ is defined as the
$\cC^0$-graph closure of $L$ acting on $\cC^3(P).$ The adjoint,
$\overline{L}^t,$ acts canonically on the dual space, i.e., the space of
regular Borel measures on $P.$ If $dV_{P}$ is a smooth non-degenerate measure
on $P,$ then $d\mu_{L}=WdV_{P}.$ For $v$ a smooth function, we then have the
relation: $\hL v=W^{-1}\overline{L}^t(Wv).$

It is clear from the discussion in~\ref{ssec4.1.006} that functions in
$\cC^3(P)$ belong to the domain of the operator $L^{\Xi}_Q.$ For
$f\in\cC^3(P),$ let $v(t)=e^{tL^{\Xi}_Q}f,$ and let $u(t)$ be the regular
solution to $(\pa_tu-Lu)=0,$ with $u(t)=f,$ given by Theorem 11.2.1
in~\cite{EpMaz2}. The regularity results in~\cite{EpMaz2} show that
$u(t)\in\Dom(L^{\Xi}_Q)$ for all $t\in [0,\infty).$ Thus, there is a constant
$m$ such that
\begin{equation}
\begin{split}
  \pa_t\|u(t)-v(t)\|^2_{L^2(P;d\mu_L)}&=2(L(u(t)-v(t)),(u(t)-v(t)))_{L^2(P;d\mu_L)}\\
&=-2\tQ^{\Xi}(u(t)-v(t),u(t)-v(t))\\
&\leq m \|u(t)-v(t)\|^2_{L^2(P;d\mu_L)}.
\end{split}
\end{equation}
The last line follows using the same argument used to prove~\eqref{eqn135.014}.
This then implies that $u(t)=v(t)$ for all $t\geq 0.$ Hence, if
$p_t(\xi,\eta)$ is kernel for $e^{tL^{\Xi}_Q},$ then the regular
solution is given by
\begin{equation}
  u(\xi,t)=\int\limits_{P}f(\eta)p_t(\xi,\eta)d\mu_{L}(\eta).
\end{equation}
That is, the heat kernel defined by the $L^2$-semigroup is the same as that
defined by the $\cC^0$-theory.

In light of~\eqref{eqn164.014} and~\eqref{eqn167.014}, the argument used to
prove Corollary~\ref{cor5.2} can easily be adapted to prove the following upper
bound on the heat kernel.
\begin{theorem}\label{cor5.2.2} Assume that $P$ is a compact manifold with
  corners and $L$ is a generalized Kimura diffusion defined on $P$ with
  positive weights. If we represent the kernel of the operator $e^{tL}$ as
  $p_t(\xi,\eta)d\mu_{L}(\eta),$ then  there  are positive constants $C_0,C_1,C_2$
  so that, for all $t>0$ and pairs $\xi,\eta\in P$ we have
\begin{equation}\label{eqn207.0055}
 p_t(\xi,\eta)\leq \frac{C_0\exp\left(-\frac{\rho_i^2(\xi,\eta)}{C_2t}\right)}
{\sqrt{\mu_{L}(B^i_{\sqrt{t}}(\xi))\mu_{L}(B^i_{\sqrt{t}}(\eta))}}
\times\left(1+\frac{\rho_i(\xi,\eta)}{\sqrt{t}}\right)^{D}
\cdot\exp(C_1t).
\end{equation}
For each $\eta\in P,$ the function $(\xi,t)\mapsto p_t(\xi,\eta)$ belongs
to $\CI(P\times (0,\infty)).$
\end{theorem}
\begin{proof}
  All statements have been proved but the last. The first equation
  in~\eqref{eqn158.012} shows that $p_{\cdot}(\cdot,\eta)$ is a weak solution to a
  parabolic equation to which Theorem~\ref{thm5.1.001} applies. This proves the
  last assertion.
\end{proof}

In a  neighborhood, $U\times U,$ of $(0,0)$ in the product coordinate chart $S_{n,m}\times
S_{n,m}\simeq S_{2n,2m},$ the heat kernel satisfies the equation
\begin{equation}
  (2\pa_t-L_{Q,\xi}-L_{Q,\eta}+V_{X,\xi}-V_{X,\eta})p_t(\xi,\eta)=-\hc(\eta)p_t(\xi,\eta)
\end{equation}
If we extend our analysis slightly to include the inhomogeneous problem, and
apply a bootstrap argument, then we can easily show that if the weights are
constant (so that $\hc$ and $V_X$ are smooth), then
\begin{equation}
  p_t(\xi,\eta)\in\CI(U\times U\times (0,\infty)).
\end{equation}
\begin{corollary}\label{cor6.1}
  Let $P$ be a manifold with corners and $L$ a generalized Kimura operator with
  smooth coefficients defined on $P.$ Assume that $L$ has constant weights
  along $\pa P.$ Let $\tQ$ be a globally defined (but possibly non-symmetric)
  Dirichlet form with measure $d\mu_L$ that defines $L$.  The heat kernel for
  $L$ has a representation as $p_t(\xi,\eta)d\mu_{L}(\eta),$ where
  \begin{equation}
    p_t\in\CI(P\times P\times (0,\infty)).
  \end{equation}
\end{corollary}

As noted in the proof of Corollary~\ref{cor5.1.008}, the spectrum of $L$ acting
on $\cC^0(P)$ agrees with its spectrum acting on $\cC^{0,\gamma}_{\WF}(P)$ for
any $0<\gamma<1.$ Hence, Corollary 12.3.3 in~\cite{EpMaz2} implies that there
is an $\theta<0$ so that
\begin{equation}
  \Spec_{\cC^0(P)}(L)\setminus\{0\}\subset\{\mu:\:\Re\mu<\theta\}.
\end{equation}
We can also conclude that the constant functions span the null-space of $L,$
and that there is a probability measure of the form $\nu=w(\eta)d\mu_{L}$ spanning
the nullspace of $L^t.$ In fact the proof of Corollary 12.3.3 in~\cite{EpMaz2}
shows that if $(\pa_t-L)u=0$ with $u(\xi,0)=f(\xi),$ then
\begin{equation}
  u(\xi,t)=\nu(f)+O(e^{\theta t}).
\end{equation}

If the weights are not constant, then the heat kernel $p_t(\xi,\eta)$ is no longer
smooth in the product space for $t>0.$ From Theorem~\ref{thm5.1.001} it follows that
$p_t(\cdot,\eta)\in\CI(P)$ for fixed $\eta$ and $t>0.$ In general it has a complicated
singularity as $\eta$ tends to $\pa P.$ The stationary distribution $\nu$ is the
push-forward of this kernel
\begin{equation}
  \nu(\eta)=\left[\int\limits_{P}p_t(\xi,\eta)d\sigma(\xi)\right] d\mu_{L}(\eta).
\end{equation}
The singularities of $p_t(\xi,\eta),$ beyond those arising from the measure
$d\mu_{L}(\eta),$ produce higher order terms in an asymptotic expansion of $\nu(\eta)$
as $q\to \pa P.$ 

To illustrate this we consider a simple 2d-case where the weights are non-constant; the
Kimura operator is:
\begin{equation}
  L=x\pa_x^2+\pa_y^2+b(y)\pa_x, \text{ with }b(y)>\beta>0\text{ and }b'(0)\neq 0,
\end{equation}
which implies that $d\mu_{L}(\tx;\ty)=\tx^{b(\ty)-1}d\tx d\ty.$ Working
formally one easily shows that the asymptotic expansion of $\nu$ takes the form:
\begin{equation}\label{eqn71.004}
  \nu(\tx;\ty)\sim
  \left[1+\sum_{j=1}^{\infty}\sum_{k=0}^{2j}\varphi_{jk}(\ty)\tx^{j}\log^k\tx\right]
d\mu_{L}(\tx;\ty),
\end{equation}
where the first few coefficients are given by
\begin{equation}
\begin{split}
  \varphi_{12}(\ty)&=-\frac{b'(\ty)^2}{b(\ty)}\\
\varphi_{11}(\ty)&=\frac{b''(\ty)}{b(\ty)}-2\left(\frac{b'(\ty)}{b(\ty)}\right)^2\\
\varphi_{10}(\ty)&=(1+b(\ty))\frac{b''(\ty)}{b(\ty)^2}-2(2+b(\ty))\frac{b'(\ty)^2}{b(\ty)^3}.
\end{split}
\end{equation}
This indicates the additional complexities one expects to see in this case.

\subsection{Eigenvalue Asymptotics}
If $L$ is a generalized Kimura diffusion, with positive weights, defined by a
possibly non-symmetric Dirichlet form on a compact manifold with corners, then
the operators $e^{tL}$ are trace class for all $t>0.$ This does not require the
weights to be constant. The estimate in~\eqref{eqn207.0055} shows that, for all
positive times, the heat kernels are square integrable with respect to the
finite measure $d\mu_{L}.$ Since $e^{tL}=e^{\frac{t L}{2}}e^{\frac{t L}{2}},$
it follows that $e^{tL}$ is a product of Hilbert-Schmidt operators and
therefore trace class.  We do not pursue the non-symmetric case further here, as
heat kernel asymptotics do not generally lead to eigenvalue asymptotics unless
the spectrum is real.

Assuming now that $L=L_Q$ is defined globally by a symmetric Dirichlet form, it
follows that the Friedrichs extension of $L_Q$ is an unbounded self adjoint
operator acting on $L^2(P;d\mu_{\bb}).$ As noted above, the trace of the heat
semigroup, $\Tr(e^{tL})$ is finite for $\Re t>0.$ From this it follows
immediately that $L$ has a compact resolvent acting on $L^2(P).$ Let
$0=\lambda_1<\lambda_2\leq\lambda_2\leq \cdots$ be the spectrum of $-L_{Q}.$
Define the counting function
\begin{equation}
  N(\lambda)=|\{i:\lambda_i\leq \lambda\}|.
\end{equation} 
Let $d=\dim P.$ The heat kernel estimates in local coordinate charts~\eqref{eqn207.005} show that
\begin{equation}
  \lim_{t\to 0^+}t^{\frac d2}\Tr(e^{tL}),
\end{equation}
exists; the standard Tauberian argument then gives asymptotics for $N(\lambda)$ as
$\lambda\to\infty.$

\begin{theorem} 
  Let $P$ be a compact manifold with corners of dimension $d,$ and $L$ a Kimura
  operator, self adjoint on $L^2(P;d\mu_P),$ defined by a globally defined symmetric
  Dirichlet form. Assume that the weights are strictly positive. The heat
  kernel $e^{tL}$ is a trace class operator, and there is a dimensional
  constant $K_d$ so that
   \begin{equation}
     \lim_{t\to 0^+}t^{\frac d2}\Tr e^{tL}=K_d\mu_{L}(P).
   \end{equation}
The counting function $N(\lambda),$ for the eigenvalues of $L,$ satisfies the
asymptotic relation
\begin{equation}
  N(\lambda)=\lambda^{\frac d2}\left(\frac{K_d\mu_{L}(P)}{\Gamma(1+d/2)}+o(1)\right).
\end{equation}
\end{theorem}
\begin{remark} Note that this theorem does not require the assumption that the
  weights are constant along the boundary. 
\end{remark}
\begin{proof}
Let $p_t(\xi,\eta)$ be the heat kernel. For $\delta>0$ define
\begin{equation}
  P_{\delta}=\{\xi\in P:\: \rho_i(\xi,\pa P)\geq \delta\}.
\end{equation}
Proposition 2.32 in~\cite{Berline-Getzler-Vergne1} implies that the trace is
given by the formula
\begin{equation}
  \Tr(e^{tL})=\int\limits_{L}p_t(\xi,\xi)d\mu_{L}(\xi).
\end{equation}
For $\delta>0,$  $L\restrictedto_{P_{\delta}}$ is uniformly elliptic with smooth
coefficients and therefore it follows from a classical argument that
\begin{equation}
  p_{t}(\xi,\xi)=\frac{K_d}{t^{\frac{d}{2}}}+O\left(\frac{1}{t^{\frac{d}{2}-1}}\right),
\end{equation}
with $t^{\frac{d}{2}}p_{t}(\xi,\xi)$ converging to $K_d$ uniformly in
$P_{\delta},$ see~\cite{Berline-Getzler-Vergne1}. Because $p_{t}(\xi,\xi)$ is
non-negative we see that
\begin{equation}
\begin{split}
  \liminf_{t\to 0^+}t^{\frac d2}\Tr(e^{tL})&\geq
  \liminf_{t\to 0^+}\int\limits_{P_{\delta}}t^{\frac d2}p_{t}(\xi,\xi)d\mu_{L}(\xi)\\
&=K_d\mu_{L}(P_{\delta})=K_d\mu_{L}(P)-a_{\delta},
\end{split}
\end{equation}
where $a_{\delta}\to 0$ as $\delta\to 0.$ 

On the other hand, the following lemma gives an estimate for
\begin{equation}
  t^{\frac d2}\int\limits_{P\setminus P_{\delta}}p_{t}(\xi,\xi)d\mu_{L}(\xi).
\end{equation}
\begin{lemma} Suppose that $d\mu_{L}$ is a measure defined on $P,$ a compact
  manifold with corners, with weights $\{b_j(\eta)\}$ bounded below by
  $\beta_0>0,$ and $\cC^1$ in the square-root variables. There is a constant
  $C$ depending on the dimension, $\beta_0,$ and the upper bound $\beta_1$ on
  $\|\nabla b_j(w,y)\|,$ so that
  \begin{equation}
    t^{\frac d2} \int\limits_{P\setminus P_{\delta}}p_{t}(\xi,\xi)d\mu_{L}(\xi)\leq C\delta
  \end{equation}
\end{lemma}
\begin{remark} As before, we could replace the regularity assumption in this
  lemma with~\eqref{regb}. For simplicity we use the bound on the gradient in
  the following argument.
\end{remark}
\begin{proof}
  To prove the lemma we cover $P\setminus P_{\delta}$ by a finite collection of
  coordinate charts, in which the Dirichlet form defining $L$ takes the form
  given in the square root coordinates $(w_1,\dots,w_n;y_1,\dots,y_m)$ with
  $m+n=d,$ in~\eqref{eqn23.009}. The measure takes the form
  \begin{equation}
    d\mu_{L}(w;y)=\prod_{i=1}^n w_i^{2b_i(w;y)-1}e^{U(w;y)}dwdy,
  \end{equation}
where $U$ is a bounded, continuous function and 
\begin{equation}
  0<\beta_0<b_i(w,y)<B,\text{ and }\sum_{j=1}^n|\pa_{w_j}b(w;y)|+
\sum_{l=1}^m|\pa_{y_l}b_i(w;y)|\leq \beta_1.
\end{equation}
Theorem~\ref{cor5.2.2} now gives 
\begin{equation}
  p_t((w;y),(w;y))\leq\frac{C}{\mu_{L}(B^i_{\sqrt{t}}(w;y))}.
\end{equation}
To prove the lemma we need to bound $\mu_{L}(B^i_{\sqrt{t}}(w;y))$ from
below. To that end we observe that there is a positive constant $\eta$ so that
\begin{multline}
  B^{p,e}_{\eta\sqrt{t}}(w;y)=\prod_{j=1}^n[(w_j-\eta\sqrt{t})\vee 0,w_j+\eta\sqrt{t}]\times
\prod_{l=1}^m[y_l-\eta\sqrt{t},y_l+\eta\sqrt{t}]\\\subset
B^i_{\sqrt{t}}(w;y)
\end{multline}
and therefore it suffices to bound $\mu_{L}(B^{p,e}_{\eta\sqrt{t}}(w;y))$
from below. 

To bound the contribution to the trace from $P\setminus P_{\delta}$ we integrate in these coordinates
over sets of the form
\begin{equation}
  \left\{[0,\delta]\times [0,\frac 12]^{n-1}\cup [0,\frac 12]\times[0,\delta]\times
  [0,\frac 12]^{n-2}\cup\cdots\cup
[0,\frac 12]^{n-1}\times[0,\delta]\right\}\times (-1,1)^{m}.
\end{equation}
Indeed, it obviously suffices to estimate the contribution from the first term:
$[0,\delta]\times [0,\frac 12]^{n-1}\times (-1,1)^m.$ We first get a lower
bound on 
\begin{multline}
  \mu_{L}(B^{p,e}_{\eta\sqrt{t}}(w;y))=\\
\int_{y_1-\eta\sqrt{t}}^{y_1+\eta\sqrt{t}}\cdots
\int_{y_m-\eta\sqrt{t}}^{y_m+\eta\sqrt{t}}\int_{(w_1-\eta\sqrt{t})\vee
    0}^{w_1+\eta\sqrt{t}}\cdots \int_{(w_n-\eta\sqrt{t})\vee
    0}^{w_n+\eta\sqrt{t}}\tw_1^{2b_1(\tw;\ty)-1}\cdots\tw_n^{2b_n(\tw;\ty)-1}d\tw d\ty
\end{multline}
Writing
\begin{equation}
  b_j(\tw;\ty)= b_j(\tw;\ty)- b_j(w;y)+ b_j(w;y),
\end{equation}
then within the domain of this integral, 
\begin{equation}
  b_j(\tw;\ty)\leq b_j(w;y)+\eta\sqrt{t}\beta_1.
\end{equation}
Since the coordinates $\{\tw_j\}$ are less than $1$ in the domain of the
integral (at least for small $t$), we have the estimate
\begin{multline}
  \mu_{L}(B^{p,e}_{\eta\sqrt{t}}(w;y))\geq\\
(2\eta\sqrt{t})^m
\int_{(w_1-\eta\sqrt{t})\vee
    0}^{w_1+\eta\sqrt{t}}\cdots \int_{(w_n-\eta\sqrt{t})\vee
    0}^{w_n+\eta\sqrt{t}}\tw_1^{2b_1(w;y)-1+\alpha}\cdots\tw_n^{2b_n(w;y)-1+\alpha}d\tw d\ty,
\end{multline}
where
\begin{equation}
  \alpha=2\eta\beta_1\sqrt{t}.
\end{equation}
This shows that
\begin{multline}
   \mu_{L}(B^{p,e}_{\eta\sqrt{t}}(w;y))\geq\\
(2\eta\sqrt{t})^m \frac{\prod_{j=1}^n[(w_j+\eta\sqrt{t})^{2b_j(w;y)+\alpha}-((w_j-\eta\sqrt{t})\vee
    0)^{2b_j(w;y)+\alpha}]}{\prod_{j=1}^n(2b_j(w;y)+\alpha)}.
\end{multline}

Using this estimate in the integral we can show that there is a constant, $C,$
depending only of the dimension, $\beta_0$ and $\beta_1$ so that 
\begin{equation}
 \int_{(-1,1)^m} \int_{0}^{\delta}\int_{0}^{\frac 12}\cdots\int_{0}^{\frac
   12}\frac{d\mu_{L}(w;y)}
{ \mu_{L}(B^{p,e}_{\eta\sqrt{t}}(w;y))}\leq \frac{C\delta t^{-n\eta\beta_1\sqrt{t}}}{t^{\frac{d}{2}}}.
\end{equation}
Since 
\begin{equation}
  \sup_{0<t<1}t^{-n\eta\beta_1\sqrt{t}}\leq e^{2n\eta\beta_1e^{-2}},
\end{equation}
the lemma follows easily from this estimate and those above.
\end{proof}
Using the lemma we see that
\begin{equation}
  \limsup_{t\to 0^+}t^{\frac d2}\Tr(e^{tL})\leq K_d\mu_{L}(P)+C\delta-a_{\delta}.
\end{equation}
Letting $\delta\to 0,$ we conclude that
\begin{equation}
  \lim_{t\to 0^+}t^{\frac d2}\Tr(e^{tL})=K_d\mu_{L}(P).
\end{equation}
Since we can rewrite the trace as
\begin{equation}
  \Tr(e^{tL})=\int\limits_{0}^{\infty}e^{-\lambda t}dN(\lambda),
\end{equation}
the Tauberian theorem, see~\cite{FellerV2}, implies that
\begin{equation}
  N(\lambda)=\lambda^{\frac d2}\left(\frac{K_d\mu_{L}(P)}{\Gamma(1+N/2)}+o(1)\right).
\end{equation}
\end{proof}

In the 1-dimensional case the operator takes the form:
\begin{equation}
  Lu=x(1-x)\pa_x^2u+b(x)\pa_xu.
\end{equation}
If $b\in\cC^1([0,1]),$ then this operator is defined by the Dirichlet form
\begin{equation}
  Q(u,v)=\int\limits_{0}^1x(1-x)\pa_xu\pa_xvx^{b_0-1}(1-x)^{b_1-1}e^{U(x)}dx,
\end{equation}
where
\begin{equation}
  b_0=b(0),\, b_1=b(1)\text{ and }\pa_x U(x)=\frac{b(x)-(b_0(1-x)-b_1 x)}{x(1-x)}.
\end{equation}
The operator can therefore be expressed as
\begin{equation}
  Lu=x^{1-b_0}(1-x)^{1-b_1}e^{-U(x)}\pa_x\left(x^{b_0}(1-x)^{b_1}e^{U(x)}\pa_x\right)u.
\end{equation}
From this formulation it is clear that the stationary distribution, defined as the
unique probability measure $\nu$ satisfying $L^t\nu=0$ is
\begin{equation}
  \nu=c_0x^{b_0-1}(1-x)^{b_1-1}e^{U(x)}dx,
\end{equation}
where $c_0$ is chosen so that $\nu([0,1])=1.$

This discussion applies equally well in higher dimensions. Suppose that $P$
is a domain in $\bbR^p,$ and that the
operator $L$ is globally defined by the Dirichlet form 
\begin{equation}
  Q(u,v)=\int\limits_{P}\langle A\nabla u,\nabla v\rangle e^{U}W_{\bb}(z)dz.
\end{equation}
We assume that in local coordinates $W_{\bb}$ takes the form
\begin{equation}
  W_{\bb}(x;y)=e^{w(x;y)}x_1^{b_1-1}\cdots x_n^{b_n-1},
\end{equation}
with $w(x;y)$ a smooth function.  Integrating by parts formally, we see that
\begin{equation}
  Lu= W_{\bb}^{-1}e^{-U}\nabla\cdot(e^{U}W_{\bb}A\nabla u),
\end{equation}
and therefore 
\begin{equation}
  L^t\nu = \nabla\cdot(e^{U}W_{\bb}A\nabla W^{-1}_{\bb}e^{-U} \nu).
\end{equation}
From this it is evident that if $\nu=W_{\bb}e^{U}dz,$ then
\begin{equation}
  L^t\nu=0.
\end{equation}
Under our usual assumption that each $b_i>\beta_0>0,$ the measure $\nu$ is
finite and therefore can be normalized to define the stationary distribution for
$L^t.$ This statement remains correct whether or not the $\{b_i\}$ are
constants, though if they are not, then the resulting operator is not a
standard generalized Kimura diffusion operator. Note that we have shown
in~\cite{EpMaz2} that, for generalized Kimura diffusion operators, there is a
unique stationary distribution whenever $\bb\geq \beta\bone>0.$ See
also~\cite[pg. 189]{Kimura1964}.

In a forthcoming paper by the first author and Camelia Pop,~\cite{EpPop2}, we give a
probabilistic approach to handling the logarithmically divergent perturbations,
and an independent proof of the Harnack inequality for generalized Kimura
diffusion operators. This paper also establishes various properties of the
Markov processes defined by these operators, and the solutions of the
corresponding systems of SDEs. Pop has further analyzed the probabilistic aspects
of Kimura diffusions in two additional papers,~\cite{Pop1, pop2}, establishing
among other things, that the Feynman-Kac and Girsanov formul{\ae} can
be used to represent solutions to these diffusion equations.

The results of this paper represent considerable progress in our understanding
of the qualitative properties of solutions to Kimura diffusion equations, its
kernel function, and the relationship between $L$ and $L^t,$ at least when the
weights are positive. The main outstanding analytic questions  seem to be:
\begin{enumerate}
\item What is the structure of the heat kernel, and what are the 
estimates for solutions of the parabolic problem when the weights vanish 
at some points of $\pa P?$ For biological applications it is
reasonable to consider cases where the weights vanish on hypersurface boundary
components, or on components of lower dimensional strata of the boundary. In
these cases the measure $d\mu_{L}$ may not be finite.
\item For non-constant weights, what is the detailed behavior of the
heat kernel near the incoming face?
\item What does the size of the gap in the spectrum around $0$ depend upon?
\item Under what conditions is the span of the eigenfunctions of $L$ dense in $\cC^0(P)?$
\end{enumerate}

\appendix
\section{Lemmas for the Proof of Theorem~\ref{prop3.2.003} }\label{appC}
The argument begins with a series of geometric lemmas,
see~\cite{SaloffCosteLMS} or~\cite{Jerison1}, which we now recall. We let
$E=B^i_r(w_0;y_0),$ and $B_0$ denote a $\rho^e_{\infty}$-ball in the set $\cF,$
which we call a covering though really $\{2B:\: B\in\cF\}$ is a covering, so
that $(w_0;y_0)\in 2B_0.$ Recall that $(w_0;y_0)$ is the center of the ball
$E.$ Let $B\in \cF$ be another ball in the covering, with center $(w;y).$ We
let $\gamma_B$ denote the Euclidean geodesic from $(w;y)$ to the center of
$B_0.$ The following geometric lemma is proved in~\cite{SaloffCosteLMS}.
\begin{lemma} For any $B\in\cF$ we have that
  \begin{equation}
    d(\gamma_B,\pa E)\geq \frac{1}{2}d(B,\pa E)=\frac{10^3}{2}r(B).
  \end{equation}
Moreover, any ball $B'$ in $\cF$ such that $2B'$ intersects $\gamma_B$ has
radius bounded below by
\begin{equation}
  r(B')\geq \frac{1}{4} r(B).
\end{equation}
\end{lemma}

An important feature of this argument is the construction of a chain of balls
in $\cF$ that join a given ball $B\in\cF$ to the central ball, $B_0.$ For each
$B\in\cF$ we let $\cF(B)=\{B_0,B_1,\dots,B_{\ell(B)-1}\},$ where
$B_{\ell(B)-1}=B,$ and
\begin{equation}
\overline{2B_i}\cap\overline{2B_{i+1}}\neq\emptyset.
\end{equation}
This chain is constructed by following the intersections of the doubles of the
balls along the geodesic $\gamma_B.$ The details are
in~\cite{SaloffCosteLMS}. There are two further geometric properties of this cover
that were already proved in~\cite{Jerison1}.
\begin{lemma}\label{lem3.4.005} For any $B\in\cF$ and two consecutive balls $B_i, B_{i+1}$ in
  $\cF(B)$  we have that $B_{i+1}\subset 4 B_i$ and the estimate
  \begin{equation}
    (1+10^{-2})^{-1}r(B_i)\leq r(B_{i+1})\leq (1+10^{-2})r(B_i).
  \end{equation}
Moreover, there is a constant $c$ independent of $B,$ so that
\begin{equation}
  \mu_{\bb}(4B_i\cap 4 B_{i+1})\geq c\max\{\mu_{\bb}(B_i),\mu_{\bb}(B_{i+1})\},
\end{equation}
where
\begin{equation}
  \mu_{\bb}(B)=\int\limits_{B}d\mu_{\bb}(\tw;\ty).
\end{equation}
\end{lemma}
\noindent
Finally
\begin{lemma}\label{lem3.5.005}
  For any ball $B\in\cF$ and any ball $A\in\cF(B)$ we have that $B\subset 10^4
  A.$
\end{lemma}

The remainder of the proof of Theorem~\ref{prop3.2.003} proceeds very much
as in~\cite{SaloffCosteLMS}, though our argument is a little simpler. For
consecutive balls $B_i, B_{i+1}$ in a chain, $\cF(B),$ we need to compare the
mean values $u_{4B_i}, u_{4B_{i+1}}.$
\begin{lemma}\label{lem3.6.005}
  Under the assumptions of Theorem~\ref{prop3.2.003}, there exists a
  constant $C$ independent of $\cF$ so that for any consecutive balls $B_i, B_{i+1}$ in a
chain, $\cF(B),$ for $B\in \cF$ we have the estimate
\begin{equation}
  |u_{4B_i}-u_{4B_{i+1}}|\leq
  C\frac{r(B_i)}{\sqrt{\mu_{\bb}(B_i)}}\left(\,\int\limits_{16B_i}|\nabla
  u|^2d\mu_{\bb}\right)^{\frac 12}.
\end{equation}
\end{lemma}
\begin{proof} Following Saloff-Coste, we write
  \begin{equation}
    \begin{split}
      &[\mu_{\bb}(4B_i\cap 4B_{i+1})]^{\frac 12}|u_{4B_i}-u_{4B_{i+1}}|  =
\left(\,\int\limits_{4B_i\cap
    4B_{i+1}}|u_{4B_i}-u_{4B_{i+1}}|^2d\mu_{\bb}\right)^{\frac 12}\\
&\leq \left(\,\int\limits_{4B_i\cap
    4B_{i+1}}|u-u_{4B_i}|^2d\mu_{\bb}\right)^{\frac
  12}+\left(\,\int\limits_{4B_i\cap
    4B_{i+1}}|u-u_{4B_{i+1}}|^2d\mu_{\bb}\right)^{\frac 12}\\
&\leq Cr(B_i) \left(\,\int\limits_{4B_i}|\nabla
  u|^2d\mu_{\bb}\right)^{\frac 12}+Cr(B_{i+1}) \left(\,\int\limits_{4B_{i+1}}|\nabla
  u|^2d\mu_{\bb}\right)^{\frac 12}.
    \end{split}
  \end{equation}
We use Corollary~\ref{cor3.1.005} to pass from the second line to the
third. The conclusion now follows from the Lemma~\ref{lem3.4.005}. 
\end{proof}

Recall that the maximal function is defined by
\begin{equation}
  M_rf(x)=\sup_{\{B:x\in B, r(B)<r\}}\frac{1}{\mu_{\bb}(B)}\int\limits_{B}|f|d\mu_{\bb}
\end{equation}
Since $d\mu_{\bb}$ is a doubling measure, it satisfies a maximal
inequality.  For $1\leq p\leq \infty,$ and $1\leq K$ there is a constant
$C(p,K)$ so that, for all $f\in\CI_0(S_{n,m})$ 
\begin{equation}
  \|M_rf\|_{L^p(S_{n,m};d\mu_{\bb})}\leq C(p,K)\|f\|_{L^p(S_{n,m};d\mu_{\bb})}.
\end{equation}
The maximal inequality has the following remarkable consequence:
\begin{lemma}\label{lem3.7.005}
  Fix $0<R,$ $1\leq K$ and $1\leq p<\infty,$ There is a constant $C(p,K)$ so that, for
  any sequence of balls $\{B_i\}$ of radius at most $R,$  and any sequence of
  non-negative numbers $\{a_i\}$ we have the estimate
  \begin{equation}
   \left\|\sum_{i}a_i\chi_{KB_i}\right\|_{L^p(S_{n,m};d\mu_{\bb})}\leq C(p,K)
\left\|\sum_{i}a_i\chi_{B_i}\right\|_{L^p(S_{n,m};d\mu_{\bb})}.
  \end{equation}
\end{lemma}
\begin{proof}[Proof of Theorem~\ref{prop3.2.003}] We can now give the final
  estimates needed to prove the theorem. Recall that $E=B_r(w_0;y_0)$ is a
  ball in $S_{n,m}$ and $\cF$ is a Jerison ``covering'' as described above. In
  fact the balls in $\cF$ are disjoint and $\{2B:\: B\in\cF\}$ is a
  covering. We let $B_0$ be a ball in $\cF$ so that $(w_0;y_0)\in 2B_0.$ For
  any ball $B\in\cF,$ we let $\cF(B)=(B_0,B_1,\dots,B_{l(B)-1})$ denote a chain
  joining $B=B_{l(B)-1}$ to $B_0.$ 

As $E=\cup_{B\in\cF}2B,$ we see that
\begin{equation}
\begin{split}
  \int\limits_{E}|u-u_{4B_0}|^2&d\mu_{\bb}\leq \sum\limits_{B\in\cF}\int\limits_{2B}|u-u_{4B_0}|^2d\mu_{\bb}
\\
&\leq 4\sum\limits_{B\in\cF}\int\limits_{4B}(|u-u_{4B}|^2+|u_{4B}-u_{4B_0}|^2)d\mu_{\bb}\\
&4\sum\limits_{B\in\cF}\left[Cr(4B)^2\int\limits_{4B}|\nabla
  u|^2d\mu_{\bb}+|u_{4B}-u_{4B_0}|^2d\mu_{\bb}(4B)\right].\end{split}
\end{equation}
We use Corollary~\ref{cor3.1.005} to pass from the second to the third
line. Since $4B\subset E,$ we can use~\eqref{eqn132.005} to conclude that there
is a constant $C_0$ so that
\begin{equation}
  \sum\limits_{B\in\cF}r(4B)^2\int\limits_{4B}|\nabla u|^2d\mu_{\bb}
\leq C_0r^2\int\limits_{E}|\nabla u|^2d\mu_{\bb}.
\end{equation}
The next step is to establish a similar estimate for
\begin{equation}
 I= \sum\limits_{B\in\cF}\int\limits_{4B}|u_{4B}-u_{4B_0}|^2d\mu_{\bb}.
\end{equation}

Using Lemma~\ref{lem3.6.005} we obtain the estimate
\begin{equation}
  |u_{4B}-u_{4B_0}|\leq \sum_{i=0}^{l(B)-1}|u_{4B_i}-u_{4B_{i+1}}|\leq
C\sum_{i=0}^{l(B)-1}\frac{r(B_i)}{\sqrt{\mu_{\bb}(B_i)}}\left(\,\int\limits_{16B_i}
|\nabla u|^2d\mu_{\bb}\right)^{\frac 12}.
\end{equation}
According to Lemma~\ref{lem3.5.005} the ball $B$ is contained in $10^4 B_i$ for
any $i$ and therefore
\begin{equation}
  |u_{4B}-u_{4B_0}|\chi_B\leq
C\sum\limits_{A\in\cF}\frac{r(A)}{\sqrt{\mu_{\bb}(A)}}\left(\,\int\limits_{16A}
|\nabla u|^2d\mu_{\bb}\right)^{\frac 12}\chi_{10^4 A}\chi_{B}.
\end{equation}
As the balls in $\cF$ are disjoint, we see that
\begin{equation}
 \sum_{B\in\cF} |u_{4B}-u_{4B_0}|^2\chi_B\leq
C\left|\sum\limits_{A\in\cF}\frac{r(A)}{\sqrt{\mu_{\bb}(A)}}\left(\,\int\limits_{16A}
|\nabla u|^2d\mu_{\bb}\right)^{\frac 12}\chi_{10^4 A}\right|^2.
\end{equation}
Once again, since the balls in $\cF$ are disjoint, we can apply
Lemma~\ref{lem3.7.005} to conclude that there is a constant $C_2$ so that
\begin{equation}
\begin{split}
 \int\limits_{E}\sum_{B\in\cF} |u_{4B}-u_{4B_0}|^2\chi_Bd\mu_{\bb}&\leq
C_2\int\limits_{E}\left|\sum\limits_{A\in\cF}\frac{r(A)}{\sqrt{\mu_{\bb}(A)}}\left(\,\int\limits_{16A}
|\nabla u|^2d\mu_{\bb}\right)^{\frac 12}\chi_{A}\right|^2d\mu_{\bb}\\
&\leq
C_2\int\limits_{E}\sum\limits_{A\in\cF}\frac{r^2(A)}{\mu_{\bb}(A)}\left(\,\int\limits_{16A}
|\nabla u|^2d\mu_{\bb}\right)\chi_{A}d\mu_{\bb}.
\end{split}
\end{equation}
The key point here is that, in the first line,  we have replaced the characteristic functions
$\{\chi_{10^4 A}\}$ in the inner sum with $\{\chi_{A}\}.$  We
pass to the second line by using the fact that the balls in $\cF$ are
disjoint. Using the doubling property one last time it follows that
\begin{equation}
  I\leq Cr^2\int\limits_{E}|\nabla u|^2d\mu_{\bb}.
\end{equation}

Combining our results thus far, we have shown that that there is a constant
$C'$ so that
\begin{equation}
  \int\limits_{E}|u-u_{B_0}|^2d\mu_{\bb}\leq C'r^2\int\limits_{E}|\nabla
  u|^2d\mu_{\bb}.
\end{equation}
The fact that the integral $\int_E|u-a|^2d\mu_{\bb},$ where $a\in\bbR,$ is minimized when
$a=u_{E},$ and the uniform equivalence of the metrics $\rho^e_{\infty}$ and
$\rho_i$ complete the proof of the theorem.
\end{proof}

\section{Lemmas for Section~\ref{sec4.2.001}}\label{appC2}
In this section we give analogues for Theorem 5.2.9, Theorem 5.2.16, Theorem
5.2.17 and Lemma 5.4.1 in~\cite{SaloffCosteLMS}. These are the ingredients
needed to apply Saloff-Coste's proof of the Harnack inequality for non-negative solutions to
$\pa_tu-(L_Q-V_x-c)u=0,$ and certain estimates for the heat kernel. These
results are largely consequences of the following two estimates: 1. for balls $B^i_{r}(x;y)$
we have the doubling estimate:
 \begin{equation}
    \mu_{\bb}(B^i_{2r}(x;y))\leq 2^D\mu_{\bb}(B^i_{r}(x;y)),
  \end{equation}
and 2. for all functions in $\cD(Q_{B^i_r(x;y)})$ we have the Sobolev inequality:
  \begin{multline}\label{SobIneq.2}
    \left[\int\limits_{B^i_r(x;y)}|u|^{\frac{2D}{D-2}}d\mu_{\bb}(\tx,\ty)\right]^{\frac{D-2}{D}}\leq\\
C_S\frac{r^2}{[\mu_{\bb}(B^i_r(x;y))]^{\frac 2D}}\left[Q_{B^i_r(x;y)}(u,u)+\frac{1}{r^2}(u,u)_{\bb}\right].
  \end{multline}

We first begin with an argument due to Moser, and appearing, in the elliptic
case, as part of the proof of Lemma 2.2.1 in~\cite{SaloffCosteLMS}, to show
that weak nonnegative subsolutions to the equation $\pa_tu=(L_Q-V_X-c)u$ are
locally bounded. We allow the vector field $V_X$ and the scalar potential $c$
to be somewhat singular. Assume that there is a $k\in\bbN,$ and a constant $M$
so that the coefficients satisfy the estimate in~\eqref{eqn220.012}.  This
generality is essential for the applications to population genetics. We give a
fairly detailed proof of this statement.

For $0<r,\delta,$ $s\in\bbR,$ and $q\in S_{n,m}$ we let
\begin{equation}
  W_r(s,q)=(s-r^2,s)\times B^i_r(q),\text{ and }W(\delta)=(s-\delta
r^2,s)\times B^i_{\delta r}(q).
\end{equation}
\begin{lemma}\label{lemB1.010}
  Assume that $\bb=(b_1,\dots,b_n)$ are positive differentiable functions of $(x;y).$ Suppose that
  $X(x;y)\in\cC^0(\Int S_{n,m};\bbR^{n+m})$ is
  constant outside of a compact set, and $c(x;y)$ is a measurable function
  supported in a compact set, both of which satisfy~\eqref{eqn220.012}.  There is a
  constant $C_1$ that depends only on the doubling dimension, $D,$ so that with
  $0<\delta<1,$ and $r<R,$ for $u$ a non-negative weak subsolution of
\begin{equation}\label{eqn222.004}
  \pa_tu=(L_Q-V_X-c)u
\end{equation}
 in $W_r(s,q),$  we have the
estimates, for $0<p:$
\begin{equation}\label{eqn245.009}        
  \sup_{W(\delta)}u^p\leq \frac{C_1}{(1-\delta)^{D+2}r^2\mu_{\bb}(B_r^i(q))}
\iint\limits_{W(1)}u^pd\mu_{\bb}dt.
\end{equation}
In particular, a weak subsolution is bounded for positive times.
\end{lemma}
\begin{remark} This is part of Theorem 2.1 in~\cite{Sturm2}
  in a more general context, but one that would require additional hypotheses
  in the present circumstance.
\end{remark}
\begin{proof}
  Recall that $u$ is a weak subsolution in $W$ if $u\in\Dom(Q)$ and for any
  non-negative function in $\Dom(Q)$ with compact support in $W,$  we have, for
  a.e. $t,$ that
  \begin{equation}
    \int\limits_{B}[u_t\varphi+\langle A\nabla u,\nabla \varphi\rangle+
\langle A\nabla u,X \varphi\rangle+cu\varphi]d\mu_{\bb}(x;y)\leq 0.
  \end{equation}
Moser's trick to prove this lemma is to use test functions $\varphi$ of the form $\varphi=\psi^2 G(u).$ 
Here $\psi\in\cC^{\infty}_c(W),$ and $G:\bbR_+\to\bbR_+$ is a piecewise $\cC^1$
function that satisfies the conditions:
\begin{enumerate}
\item $G(r)=ar,$ $a$ a positive constant, when $r$ is sufficiently large.
\item $G'(r)$ is non-negative and non-decreasing, which implies
\item $G(r)\leq rG'(r).$
\end{enumerate}
We then set $H'(r)=\sqrt{G'(r)},$ with $H(0)=0.$ Once again, the mean value
theorem implies that $H(r)\leq rH'(r).$ Finally we set
\begin{equation}
  K(s)=\int\limits_{0}^r G(\rho)d\rho,
\end{equation}
and assume that there is a universal constant, $C_0,$ so that 
\begin{equation}\label{eqn225.010}
H^2(r)\leq C_0 K(r),
\end{equation}
 for $r\in [0,\infty).$ 

Since $G(r)$ grows linearly for large $r,$ the function $\varphi=\psi^2 G(u)$
is an admissible test function, and therefore:
\begin{equation}\label{eqn226.011}
    \int\limits_{B}[\pa_tK(u)\psi^2+\langle A\nabla u,\nabla \psi^2 G(u)\rangle+
\langle A\nabla u,X \psi^2G(u)\rangle+cu\psi^2G(u)]d\mu_{\bb}(x;y)\leq 0.
  \end{equation}
Using the argument on page 39 of~\cite{SaloffCosteLMS} we can show that this
implies that there is a constant $C_1$ so that
\begin{multline}\label{eqn226.010}
    \int\limits_{B}[\pa_tK(u)\psi^2+\frac{1}{4}\langle A\nabla\psi H(u),\nabla
    \psi H(u)\rangle]d\mu_{\bb}\leq \\
C_1\int\limits_{B}[(\langle A\nabla\psi,\nabla\psi\rangle +\psi^2)u^2G'(u)+
(\langle AX,X\rangle+|c|) \psi^2uG(u)]d\mu_{\bb}.
\end{multline}
From this point the argument would be standard, but for the fact that the
$(\langle AX,X\rangle+|c|)$-term is not required to be bounded near the
boundary. Below we prove a lemma that allows us to handle the contribution from
near the singular locus. 

Following Moser, and Saloff-Coste, we take:
\begin{equation}
  H_N(r)=\begin{cases} &r^{\frac{p}{2}}\text{ for }0\leq r\leq N\\
&N^{\frac{p}{2}-1}r\text{ for } r> N,\end{cases}
\end{equation}
which implies that
\begin{equation}
  G_N(r)=\begin{cases} &\frac{p^2}{4(p-1)}r^{p-1}\text{ for }0\leq r\leq N\\
&N^{p-2}(r-N)+\frac{p^2}{4(p-1)}N^{p-1}\text{ for } r> N,\end{cases}
\end{equation}
and
\begin{equation}
  K_N(r)=\begin{cases} &\frac{p}{4(p-1)}r^{p}\text{ for }0\leq r\leq N\\
&\frac{N^{p-2}}{2}(r-N)^2+\frac{p^2}{4(p-1)}N^{p-1}(r-N)+\frac{p}{4(p-1)}N^{p}\text{ for } r> N.\end{cases}
\end{equation}

We now show that there is constant $C_0$ independent of $N$ and $2\leq p,$ so
that
\begin{equation}\label{eqn230.010}
  \frac{K_N(r)}{H_N^2(r)}\geq C_0.
\end{equation}
For $0\leq r\leq N,$ we have that
\begin{equation}
  \frac{K_N(r)}{H_N^2(r)}=\frac{p}{4(p-1)}\geq \frac{1}{4}.
\end{equation}
For $r>N,$ we let $r=N\rho,$ and obtain that
\begin{equation}
  \frac{K_N(N\rho)}{H_N^2(N\rho)}=
\frac{\frac{(\rho-1)^2}{2}+\frac{p^2}{4(p-1)}(\rho-1)+\frac{p}{4(p-1)}}{\rho^2},
\end{equation}
from which is it clear that the minimum does not depend on $N.$ A simple
calculation shows that the minimum on $[1,\infty)$ is assumed at $\rho=1,$ and
therefore~\eqref{eqn230.010} holds for $2\leq p,$ with $C_0=\frac{1}{4}.$

The main new result is in the following lemma:
\begin{lemma}\label{lemB2.010} Assume that $\bb=(b_1,\dots,b_n)$ are positive differentiable
  functions of $(\bx;\by),$ with $0<\beta_0<b_j,$ constant outside a compact set. Let $q$ be a
  measurable function defined on $S_{n,m}$ that satisfies
\begin{equation}\label{eqn220.0122}
    |q(x;y)|\leq M\left[\chi_{B}(x;y)\sum_{j=1}^n|\log x_i|^k+1\right].
  \end{equation}
for some $k\in\bbN, 0<M,$ and  $B$ a bounded neighborhood of $(0;0),$ 
Given $\eta>0$ there  is a $C_{\eta}$ so that  we have
\begin{multline}\label{eqn263.012}
  \int\limits_{S_{n,m}}|q(x;y)|u G(u)d\mu_{\bb}\leq
 \eta\int\limits_{S_{n,m}}\langle A\nabla u,\nabla
 u \rangle G'(u) d\mu_{\bb}+\\
C_{\eta} \int\limits_{S_{n,m}}u^2G'(u)d\mu_{\bb},
\end{multline}
for $u$ a bounded, non-negative, compactly supported  function in $\Dom(Q).$
\end{lemma}

We give the proof of the lemma below. With this bound we can estimate the
contribution of the last term on the right hand side of~\eqref{eqn226.010} near
the boundary and thereby show that there is a constant, independent of $2\leq
p,$ and $N$ so that
\begin{multline}
    \int\limits_{W}[\pa_tK_N(u)\psi^2+\langle A\nabla \psi H_N(u),\nabla \psi
    H_N(u)\rangle] d\mu_{\bb}\leq\\
C \int\limits_{W}[\langle A\nabla \psi,\nabla\psi\rangle+\psi^2]u^2G_N'(u)d\mu_{\bb}.
  \end{multline}
We  first argue as on page 40, and then as on page 121-2, using the fact
that $H_N(u)\leq 4K_N(u).$  Letting $N\to\infty$ we conclude that,
with $\theta=1+2/D,$ we have:
\begin{equation}
  \iint\limits_{W(\delta)}u^{p\theta}d\mu_{\bb}dt\leq
  \frac{C}{(\delta'-\delta)^{2+D}r^2\mu_{\bb}(B_r)}
\left(\frac{A p^2}{r^2(\delta'-\delta)^{2}}\iint\limits_{W(\delta')}u^pd\mu_{\bb}dt\right)^{\theta}
\end{equation}
Starting with $p=2$ we can iterate this inequality to conclude that $u^p$ is
integrable for any $2\leq p,$ and then apply the argument on page 122
of~\cite{SaloffCosteLMS} to complete the proof of Lemma~\ref{lemB1.010} for
$p=2.$

In~\cite{SaloffCosteLMS} the fact that $u^p$ is a subsolution if $p>1$ is
employed to use the argument above to complete the proof of the lemma. Since we
are allowing lower order terms, we cannot use this argument and use instead an
argument given in~\cite{Moser2}.  We do not give the complete proof, but
demonstrate that an exact analogue of the last formula on page 737
of~\cite{Moser2} holds in the present context. From that point onward, the
conclusion then follows, as in Moser, by employing the Sobolev
inequality~\eqref{SobIneq}, which holds for $Q_B.$ From the $p=2$ case we can
assume that $u$ is bounded, and therefore $u^p\in\cD(Q)$ if $p>1.$

We  let $\varphi=u^{p-1}\psi^2$ in~\eqref{eqn226.011}.
Here $\psi$ is a smooth function that is $1$ in $W(\delta)$ and zero outside
of $W(1).$ A little algebra shows that the condition
in~\eqref{eqn226.011} can be re-expressed as:
\begin{multline}
  \int\left[\frac{1}{4}\pa_t(v^2)\psi^2+
\left(1-\frac{1}{p}\right)\langle A\nabla v,\nabla
v\rangle\psi^2\right]d\mu_{\bb}=\\
-\int\left[\langle A\nabla v,\nabla
\psi\rangle v\psi+\frac{1}{2}\langle A\nabla v, X\rangle v\psi^2+cv^2\psi^2\right]d\mu_{\bb},
\end{multline}
where $v=u^{\frac p2}.$

We let 
\begin{equation}
  \epsilon=\min\left\{\frac{1}{4}\left|1-\frac 1p\right|,\frac{1}{4}\right\}.
\end{equation}
The Cauchy-Schwarz and arithmetic-geometric inequalities show that
\begin{equation}
  |\langle A\nabla v,\nabla \psi\rangle v\psi|\leq
\frac{1}{4\epsilon}\langle A\nabla \psi,\nabla
\psi\rangle v^2 +\epsilon\langle A\nabla v,\nabla v\rangle \psi^2
\end{equation}
and
\begin{equation}
  |\langle A\nabla v,X\rangle v\psi^2|\leq
\frac{1}{4\epsilon}\langle A X,X\rangle \psi^2v^2 +
\epsilon\langle A\nabla v,\nabla v\rangle \psi^2.
\end{equation}

This demonstrates that, for $2<p,$
\begin{multline}
  \frac{1}{4}\int \pa_t(\psi^2v^2)d\mu_{\bb}+2\epsilon\int\langle
  A\nabla v,\nabla v\rangle \psi^2d\mu_{\bb}\\
\leq \frac{1}{4\epsilon}\int\left[\langle A\nabla \psi,\nabla
\psi\rangle +\langle A X,X\rangle \psi^2+4\epsilon |c|\psi^2\right]v^2d\mu_{\bb}+\\
\frac{1}{2}\int |\psi\psi_t|v^2d\mu_{\bb}.
\end{multline}
From this point the argument goes very much as in the $p=2$ case. In
particular, we use Lemma~\ref{lemB2.2.010} to control
the $q\psi^2v^2$ term, where $q=[\langle A X,X\rangle+4\epsilon |c|].,$   obtaining the estimate
\begin{multline}\label{eqn273.011}
  \frac{1}{4}\int \pa_t(\psi^2v^2)d\mu_{\bb}+\epsilon\int\langle
  A\nabla v,\nabla v\rangle \psi^2d\mu_{\bb}\\
\leq \frac{C}{4\epsilon}\int\left[\langle A\nabla \psi,\nabla
\psi\rangle +\psi^2\right]v^2d\mu_{\bb}+
\frac{1}{2}\int |\psi\psi_t|v^2d\mu_{\bb}.
\end{multline}
After integrating in $t$ this is essentially the same as the estimate at the bottom
of page 737 in~\cite{Moser2}, completing the proof of the lemma for $2\leq p.$

To obtain the estimate for $0<p<2$ we employ the argument used to prove Theorem
2.2.3 in~\cite{SaloffCosteLMS}.
\end{proof}
We still need to prove Lemma~\ref{lemB2.010}:
\begin{proof}[Proof of Lemma~\ref{lemB2.010}]
  If $q$ were bounded, i.e. $B=\emptyset,$ then the estimate
  in~\eqref{eqn263.012}, with $\eta=0,$ would follow from the fact that
  $G(u)\leq uG'(u).$ To treat the case where $B\neq\emptyset,$ we begin with a
  local version of the lemma:
\begin{lemma}\label{lemB2.01022} Assume that $\bb=(b_1,\dots,b_n)$ are positive differentiable
  functions of $(x;y),$ with $0<\beta_0<b_i,$ constant outside a compact
  set. Let $q$ be a measurable function defined on $S_{n,m}$ so that for some
  $M, B$ and $0<k$ it satisfies the estimate~\eqref{eqn220.0122}.
  Let $(x_0;y_0)\in\pa S_{n,m}.$ Given $\eta>0$ there is a open
  neighborhood 
$$U_{\delta}(x_0;y_0)=\{(x;y)\in S_{n,m}:\: |x_i-x_{0i}|<\delta,\, y_j\in (-1,1)\},$$  
so that if $\supp\chi\subset U_{\delta}(x_0;y_0),$ then there is a
$C_{\eta},$ independent of $u, G,$ and $\chi$ so that
\begin{multline}\label{eqn275.1.012}
  \int\limits_{S_{n,m}}\chi^2(x;y)|q(x;y)|uG(u)d\mu_{\bb}\leq
 \eta\int\limits_{S_{n,m}}\langle A\nabla u,\nabla
 u\rangle G'(u)\chi^2d\mu_{\bb}+\\
C_{\eta} \int\limits_{S_{n,m}}[\langle A\nabla \chi,\nabla\chi\rangle+\chi^2]u^2G'(u)d\mu_{\bb},
\end{multline}
for $u$ a non-negative function in $\Dom(Q).$
\end{lemma}
\begin{proof}[Proof of Lemma~\ref{lemB2.01022}] By relabeling, we can assume
  that
  \begin{equation}
    x_{01}=\cdots=x_{0l}=0\text{ and }0<x_{0i}\text{ for }i=l+1,\dots,n.
  \end{equation}
The key observation is that for any $\eta'>0,$ and $a>0,$ there is a $\delta_0>0$
so that if $|x_i-x_{0i}|<\delta_0,$ with $0<x_i,$ for $i=1,\dots,n,$ then
\begin{equation}\label{eqn237.010}
  M[\sum_{i=1}^n|\log x_i|^k+1]\leq \eta'\sum_{i=1}^nx_i^{-a}.
\end{equation}
For each $i=1,\dots,n,$ let $\ub_i(x;y)=\min\{b_i(x;y):x_i\in[0,x_{0i}+\delta_0]\};$
these are Lipschitz functions.  Fix a positive number
$0<a<\min\{\beta_0/2,1/4\},$ and let $0<\delta_0<1/2$ be fixed so
that~\eqref{eqn237.010} holds, and
\begin{equation}\label{eqn225.013}
  a+b_i(x;y)-\ub_i(x;y)<\frac{1}{2}\text{ for }(x;y)\in U_{\delta_0}(x_0;y_0).
\end{equation}
We may need to reduce $\delta$ several times, but we first assume that $\chi$
is supported in the set $U_{\delta_0}(x_0;y_0).$ Under this assumption
we see that
\begin{equation}\label{eqn238.010}
  \int\chi^2quG(u)d\mu_{\bb}\leq
  \eta'\int\chi^2\left[\sum_{i=1}^nx_i^{-a}\right]uG(u)d\mu_{\bb}
\leq \eta'\int\chi^2\left[\sum_{i=1}^nx_i^{-a}\right]uG(u)d\mu_{\ubb},
\end{equation}
where $d\mu_{\ubb}=x_1^{\ub_1}\cdots x_n^{\ub_n}dxdy.$

Recalling the form of the measure, and the fact that each $\ub_i$ is Lipschitz
and independent of $x_i$ for $0<x_i<\delta_0,$ we can integrate by parts to
obtain that
\begin{equation}
\begin{split}
  \int\chi^2&\left[\sum_{i=1}^nx_i^{-a}\right]uG(u)d\mu_{\ubb}=\sum_{i=1}^n\left[
\int \frac{x_i^{b_i-a}}{\ub_i-a}\pa_{x_i}\left(\chi^2 uG(u)\prod_{k\neq
    i}x_k^{\ub_k}\right)dxdy\right]\\
&=\sum_{i=1}^n\left[
\int \frac{x_i^{1-(a+b_i-\ub_i)}}{\ub_i-a}\left(\pa_{x_i}[\chi^2 uG(u)]+\chi^2 uG(u)\left(\sum_{k\neq
    i}\pa_{x_i}\ub_k\log x_k\right)\right)d\mu_{\bb}\right]
\end{split}
\end{equation}
Choosing a $0<\delta_1\leq \delta_0,$ we can arrange to have
\begin{equation}
\frac{2}{\beta_0}\sum_{i=1}^n\sum_{k\neq i}  |\pa_{x_i}\ub_k\log x_k|\leq
\frac{1}{2}\sum_{i=1}^nx_i^{-a}\text{ if }(x;y)\in U_{\delta_1}(x_0;y_0).
\end{equation}
Now assuming that $\supp\chi\subset U_{\delta_1}(x_0;y_0)$ these
inequalities and~\eqref{eqn225.013} imply that
\begin{equation}
  \int\chi^2\left[\sum_{i=1}^nx_i^{-a}\right]uG(u)d\mu_{\bb}\leq 
\frac{4}{\beta_0}
\int\left[\sum_{i=1}^n x_i^{\frac 12}\left|\pa_{x_i}[\chi^2 uG(u)]\right|\right]d\mu_{\bb}.
\end{equation}
Using the Cauchy-Schwarz inequality and the properties of $G$ we see that
\begin{multline}
  \sum_{i=1}^n x_i^{\frac 12}\left|\pa_{x_i}[\chi^2 uG(u)]\right|\leq\\
    [n\chi^2u^2G'(u)]^{\frac{1}{2}}\left[
\left(\sum_{i=1}^nx_i(\pa_{x_i}\chi)^2 u^2G'(u)\right)^{\frac{1}{2}}+
\left(\sum_{i=1}^nx_i(\pa_{x_i}u)^2 \chi^2G'(u)\right)^{\frac{1}{2}}\right].
\end{multline}
From  ellipticity hypotheses on $q(\nabla u),$ see~\eqref{eqn74},
it is clear that there is a constant $M_1$ so that for data, $f,$ supported in
a fixed small neighborhood of $\pa S_{n,m}$ we have the estimate
\begin{equation}
  \sum_{i=1}^nx_i(\pa_{x_i} f)^2\leq M_1\langle A\nabla f,\nabla f\rangle.
\end{equation}
 and therefore
 \begin{multline}
   \sum_{i=1}^n x_i^{\frac 12}\left|\pa_{x_i}[\chi^2 uG(u)]\right|\leq\\
[nM_1\chi^2u^2G'(u)]^{\frac{1}{2}}\left[(\langle A\nabla\chi,\nabla\chi\rangle
u^2G'(u))^{\frac 12}+(\langle A\nabla u,\nabla u\rangle
\chi^2G'(u))^{\frac 12}\right]
 \end{multline}
 Choosing $0<\eta'$ sufficiently small, the assertion of the lemma follows
 easily from this estimate, the arithmetic-geometric mean inequality
 and the initial estimate~\eqref{eqn238.010}.
\end{proof}

The proof of Lemma~\ref{lemB2.010} follows from the local result and a simple
covering argument. The set $B,$ appearing in the estimate~\eqref{eqn220.0122} is
compact. For each $(u;v)\in\pa S_{n,m}\cap B,$ the local result
provides an open set $U_{\delta}(u;v)$ in which the estimate~\eqref{eqn275.1.012}
holds. By compactness a finite collection $\{U_{\delta_i}(u_i,v_i):\:i=1\dots,
I\}$ covers $\pa S_{n,m}\cap B.$ Since $\min\{\delta_1,\dots,\delta_I\}>0,$ we
can choose a collection of smooth functions $\{\chi_i\}$ with
$\supp\chi_i\subset U_{\delta_i}(u_i,v_i),$ and
\begin{equation}
 \tchi^2(x;y)=  \sum_{i=1}^I\chi_i^2(x;y)=1
\end{equation}
in a neighborhood of $B\cap \pa S_{n,m}.$ We therefore obtain that
\begin{multline}
 \sum_{i=1}^I \int\limits_{S_{n,m}}q(x;y)\chi_i^2uG(u)d\mu_{\bb}<\\
 \sum_{i=1}^I\left[\eta\int\limits_{S_{n,m}}
\langle A\nabla u,\nabla u\rangle G'(u)\chi_i^2d\mu_{\bb} +C_{\eta}\int\limits_{S_{n,m}}\left[\langle
  A\nabla\chi_i,\nabla\chi_i\rangle+\chi_i^2 \right]u^2G'(u)d\mu_{\bb}\right].
\end{multline}
Note that $C_{\eta}$ depends only on $q.$ With these choices,  $q$ is
bounded in the $\supp(1-\tchi^2),$ which completes the proof of the lemma.

\end{proof}

The proof of this lemma is easily adapted to prove the following result:
\begin{lemma}\label{lemB2.2.010} Assume that $\bb=(b_1,\dots,b_l)$ are positive differentiable
  functions of $(x;y),$ with $0<\beta_0<b_j,$
  constant outside a compact set. Let $q$ be a measurable function defined on $S_{n,m}$ that
  satisfies~\eqref{eqn220.0122}, for some $k\in\bbN,$ $B$ and $M.$
Given $\eta>0$ there is a $C_{\eta}$ so that for any
$2\leq p,$ we have
\begin{equation}
  \int\limits_{S_{n,m}}|q(x;y)|u^p(x;y)d\mu_{\bb}\leq
 \eta\int\limits_{S_{n,m}}\langle A\nabla u^{\frac{p}{2}},\nabla
 u^{\frac{p}{2}}\rangle d\mu_{\bb}+
C_{\eta} \int\limits_{S_{n,m}} u^pd\mu_{\bb},
\end{equation}
for $u$ a bounded, non-negative, compactly supported function in $\Dom(Q).$
\end{lemma}
\begin{remark}  It is evident that we
  can actually prove these results for somewhat more singular potentials,
  i.e. those satisfying an estimate of the form
  \begin{equation}
    |q(x;y)|\leq M\left[\sum_{i=1}^nx_i^{-a}+1\right],
  \end{equation}
for a constant $a<\beta_0.$
\end{remark}

We have an estimate for supersolutions, which is the analogue of
Saloff-Coste's Theorem 5.2.16 (or (2.11.c) in~\cite{Sturm2}):
\begin{lemma}\label{lemc2.01}
  Assume that $\bb=(b_1,\dots,b_n)$ are positive differentiable functions of $(x;y),$
  satisfying~\eqref{regb}, which are constant outside a compact set, $X(x;y)$
  is a continuous $\bbR^{n+m}$-valued function, satisfying~\eqref{eqn220.0122}, which is constant outside of a
  compact set, and $c(x;y)$ is a bounded measurable function supported in a
  compact set.  There is a constant $C(p,D)$ that depends on $0<p,$ the
  doubling dimension, $D,$ and Sobolev constant so that with $0<\delta<1,$ and
  $0<r,$ and $u$ a bounded, positive, weak supersolution of
\begin{equation}\label{eqn222.006}
  \pa_tu=(L_Q-V_X-c)u
\end{equation}
in $W_r(s,q),$ satisfies the estimate:
\begin{equation}\label{eqn245.010}
  \sup_{W(\delta)}u^{-p}\leq \frac{C(p,D)}{(1-\delta)^{D+2}r^2\mu_{\bb}(B^i_r(q))}\iint\limits_{W(1)}u^{-p}dt d\mu_{\bb}, \text{ for
a  }p>0.
\end{equation}
\end{lemma}
\begin{remark} For this result we need to assume that $c$ is bounded, for otherwise we
  could not begin the argument below by assuming that the supersolution is strictly positive.
\end{remark}
\begin{proof}
  Provided that $c$ is non-negative, we can assume, by replacing $u$ by
  $u+\epsilon,$ that $u$ is strictly positive. If $c$ assumes negative values,
  then we first replace $u$ by $e^{\mu t}u,$ where $\mu>\|c\|_{L^{\infty}},$
  which reduces us to the previous case.  For any non-negative function
  $\varphi$ with compact support in $B^i_{r}(q),$ a weak supersolution
  satisfies
\begin{equation}\label{eqn208.005}
  \int\left[u_t\varphi+\langle A\nabla u,\nabla\varphi+\varphi
    X\rangle+cu\varphi\right]d\mu_{\bb}\geq 0.
\end{equation}
If we let $\varphi=p\psi^2u^{-p-1},$ and set $v=u^{-\frac p2},$ then this is
equivalent to
\begin{multline}
  -\int\bigg[\psi^2\pa_t v^2+\frac{4(p+1)}{p}\psi^2\langle A\nabla v,\nabla
    v\rangle+4\langle A\nabla v,\psi\nabla
    \psi\rangle v+\\ 2\psi^2\langle A\nabla v, X
    v\rangle-pc\psi^2v^2\bigg]d\mu_{\bb}\geq 0.
\end{multline}
Once again, using Lemma~\ref{lemB2.2.010} and  the Cauchy Schwarz and arithmetic-geometric
mean inequalities, we show that there is a constant $M$ for which
\begin{multline}
  \int\left[\psi^2\pa_t v^2+\left(2+\frac{4}{p}\right)\psi^2\langle A\nabla v,\nabla
    v\rangle\right]d\mu_{\bb}\leq\\
M\int\left[(p+1)\psi^2+\langle A\nabla\psi,\nabla
    \psi\rangle\right]v^2d\mu_{\bb}.
\end{multline}
Arguing as above, we see that there is another constant $M'$ so that
\begin{multline}
  \int\left[\psi^2\pa_t v^2+\langle A\nabla \psi v,\nabla
    \psi v\rangle\right]d\mu_{\bb}\leq\\
M'(p+1)\|\langle A\nabla\psi,\nabla \psi\rangle\|_{L^{\infty}}
\int\limits_{\supp\psi}v^2d\mu_{\bb}.
\end{multline}
The statement of the lemma now follows from the iteration argument using the
Sobolev inequality given on page 129 of~\cite{SaloffCosteLMS}.
\end{proof}

To complete the argument we need to show that the appropriate analogue of
Saloff-Coste's Lemma 5.4.1 (which is Moser's Lemma 2) holds. In the present case this reads:
\begin{lemma}\label{lemc3.01}
  Assume that $\bb=(b_1,\dots,b_n)$ are positive differentiable functions of $(x;y),$
  satisfying~\eqref{regb}, which are constant outside a compact set, $X(x;y)$
  is a continuous $\bbR^{n+m}$-valued function, satisfying~\eqref{eqn220.0122},
  which is constant outside of a compact set, and $c(x;y)$ is a bounded
  measurable function supported in a compact set.  For any weak positive
  supersolution u of~\eqref{eqn222.004} in $W_r(s,q),$
  $0<r<R,$ $0<\eta<1,$ and $0<\delta<1,$ there is a constant $a(\eta,u)$ so
  that for all $0<\lambda$ we have the estimates
\begin{equation}
\begin{split}
  &\mu_{\bb}\times dt\{(x;y,t)\in W_+:\:\log u<-\lambda-a\} \leq C\frac{r^2\mu_{\bb}(B^i_r)}{\lambda}\\
&\mu_{\bb}\times dt\{(x;y,t)\in W_-:\:\log u>\lambda-a\}\leq
C\frac{r^2\mu_{\bb}(B^i_r)}{\lambda},
\end{split}
\end{equation}
where $W_+=(s-\eta r^2,s)\times B^i_{\delta r}$ and $W_-=(s-r^2, s-\eta
r^2)\times B^i_{\delta r}.$  Here $C$ is independent of $\lambda>0, s,$ and $r.$
\end{lemma}
\begin{proof} As with the proof of the previous lemma, we can assume that $u$
  is strictly positive and  show that the
  fundamental inequality used in Saloff-Coste's proof holds in this case as
  well. The proof in~\cite{SaloffCosteLMS} is contained on pages 143-145. We
  start, as before, with~\eqref{eqn208.005} and use the test function
  $\varphi=\psi^2(x;y)u^{-1}.$ Letting $v=-\log u,$ this equation takes the
  form
  \begin{equation}
    \int \left[v_t\psi^2+\langle A\nabla v,\nabla v\rangle+
\langle A\nabla v,\nabla \psi^2\rangle+\langle A\nabla v,\psi^2 X\rangle
-c\psi^2\right]d\mu_{\bb}\leq 0.
  \end{equation}
Using the Cauchy-Schwarz and  arithmetic-geometric mean inequalities we obtain
that
\begin{multline}\label{eqn243.009}
  \pa_t\int\psi^2vd\mu_{\bb}+\frac{1}{2}\int \langle A\nabla v,\nabla v\rangle d\mu_{\bb}
    \\
\leq
8\int\left[ \langle A\nabla \psi,\nabla \psi\rangle + 
\psi^2(\langle A X,X\rangle+|c|)\right] d\mu_{\bb}.
\end{multline}
We once again use Lemma~\ref{lemB2.2.010} to control the $\langle A
X,X\rangle$-term and show that there is a constant $C'$ independent of $\psi,
v, W, \delta,$ so that
\begin{multline}\label{eqn243.010}
  \pa_t\int\psi^2vd\mu_{\bb}+\frac{1}{2}\int \langle A\nabla v,\nabla v\rangle d\mu_{\bb}
    \\
\leq
C'\int\left[ \langle A\nabla \psi,\nabla \psi\rangle + 
\psi^2\right] d\mu_{\bb}.
\end{multline}

For $\psi$ we use the function $\psi(\tx;\ty)=(1-\rho_i((\tx;\ty),(x;y))/r)_+.$
As $r<R,$ it is clear that there is a constant $C$ so that the right hand side
in~\eqref{eqn243.009} is bounded by
\begin{equation}
  C\|\langle A\nabla \psi,\nabla \psi\rangle\|_{L^{\infty}}\mu_{\bb}(\supp\psi).
\end{equation}
These estimates therefore imply that
\begin{equation}
   \pa_t\int\psi^2vd\mu_{\bb}+\frac{1}{2}\int \langle A\nabla v,\nabla v\rangle d\mu_{\bb}
    \leq C\|\langle A\nabla \psi,\nabla \psi\rangle\|_{L^{\infty}}\mu_{\bb}(\supp\psi),
\end{equation}
which is equivalent to the estimate (5.4.1) in~\cite{SaloffCosteLMS}. Replacing
  Theorem 5.3.4 (The Weighted Poincar\'e Inequality) with our Proposition~\ref{prop3.4}
(Sturm's Corollary 2.5), we complete the proof of this lemma exactly as
in~\cite{SaloffCosteLMS}. The only other ingredient used in the proof is the
doubling property of the measure.
\end{proof}

\begin{small}

\end{small}
\end{document}